\documentclass[letterpaper, oneside]{amsart}
\usepackage{layout}	

\usepackage[
]{geometry}

\usepackage[parfill]{parskip}   
\usepackage{amssymb}
\usepackage{amsthm}
\usepackage{amsmath}
\usepackage{enumitem}
\usepackage{mathtools}
\usepackage{hyperref}
\usepackage{colonequals}
\usepackage{mathrsfs}
\usepackage{color}
\usepackage{subcaption}
\usepackage{adjustbox}
\usepackage{bm}
\usepackage{bbm}
\usepackage{tensor}
\usepackage{esvect}
\usepackage{mathdots}
\usepackage{tabmac}
\usepackage{longtable}

\usepackage{tikz}
\usetikzlibrary{cd, patterns, math, arrows}

\usepackage{ytableau}
\ytableausetup{centertableaux}

\theoremstyle{plain}
\newtheorem{thm}{Theorem}[section]
\newtheorem{lem}[thm]{Lemma}

\newtheorem{cor}[thm]{Corollary}

\theoremstyle{definition}
\newtheorem{defn}[thm]{Definition}

\newtheorem{example}[thm]{Example}

\theoremstyle{remark}
\newtheorem*{rmk}{Remark}

\numberwithin{equation}{section}

\newcommand{\bN}{\mathbb{N}}

\newcommand{\bZ}{\mathbb{Z}}

\newcommand{\cA}{\mathcal{A}}

\newcommand{\fP}{\mathfrak{P}}
\newcommand{\fC}{\mathfrak{C}}
\newcommand{\fQ}{\mathfrak{Q}}

\newcommand{\Sym}{\mathcal{S}}
\newcommand{\affSn}{\overline{\Sym_n}}
\newcommand{\Sinf}{\Sym_\infty}
\newcommand{\extSn}{\widetilde{\Sym_n}}
\newcommand{\extS}[1]{\widetilde{\Sym_{#1}}}

\newcommand{\offset}{\vv{\mathbf{s}}}
\newcommand{\shift}{\bm{\omega}}
\newcommand{\tsc}{{\underline{c}}}
\newcommand{\lc}{\Gamma}

\newcommand{\evac}{\textnormal{evac}}
\newcommand{\pr}{\textnormal{pr}}

\newcommand{\rw}{\textnormal{rw}}
\newcommand{\fw}{\textnormal{fw}}
\newcommand{\bk}{\textnormal{bk}}
\newcommand{\st}{\mathfrak{st}}

\newcommand{\SYT}{\textnormal{SYT}}
\newcommand{\RSYT}{\textnormal{RSYT}}

\newcommand{\SSYT}{\textnormal{SSYT}}

\newcommand{\rot}{\mathcal{R}}

\newcommand{\crho}{\tensor[_\bullet]{\rho}{}}

\DeclareMathOperator{\sh}{\textnormal{sh}}

\DeclareMathOperator{\RSK}{\textnormal{RSK}}
\DeclareMathOperator{\sgn}{\textnormal{sgn}}

\DeclareMathOperator{\im}{\textnormal{im}}

\DeclareMathOperator{\lch}{\textnormal{lch}}

\DeclareMathOperator{\con}{\mathord{{}^\frown}}

\newcommand{\su}[1]{{\mid}_{#1}}

\newcommand{\br}[1]{\left\langle{#1}\right\rangle}
\newcommand{\tbr}[1]{\mathord{\left\langle{#1}\right\rangle}}

\newcommand{\ceil}[1]{\left\lceil{#1}\right\rceil}

\title{Sign insertion and Kazhdan-Lusztig cells of affine symmetric groups}
\author{Dongkwan Kim} 
\address{School of Mathematics\\
  University of Minnesota Twin Cities\\
  Minneapolis, MN 55455\\
  U.S.A.}
\email{kim00657@umn.edu}
\author{Pavlo Pylyavskyy}
\address{School of Mathematics\\
  University of Minnesota Twin Cities\\
  Minneapolis, MN 55455\\
  U.S.A.}
\email{ppylyavs@umn.edu}
\date{\today}							

\begin{document}
\begin{abstract}  Combinatorics of Kazhdan-Lusztig cells in affine type $A$ was originally developed by Lusztig, Shi, and Xi. Building on their work, Chmutov, Pylyavskyy, and Yudovina introduced the affine matrix-ball construction (abbreviated AMBC) which gives an analog of Robinson-Schensted correspondence for affine symmetric groups. An alternative approach to Kazhdan-Lusztig theory in affine type $A$ was developed by Blasiak in his work on catabolism. He introduced sign insertion algorithm and conjectured that if one fixes the two-sided cell, the recording tableau of sign insertion process determines uniquely and is determined uniquely by the left cell. In this paper we unite these two approaches by proving Blasiak's conjecture. In the process, we show that certain new operations we introduce called partial rotations connect the elements in the intersection of a left cell and a right cell. Lastly, we investigate the connection between Blasiak's sign insertion and the standardization map acting on the set of semi-standard Young tableaux defined by Lascoux and Sch\"utzenberger.
\end{abstract}

\setcounter{tocdepth}{1}
\maketitle

\renewcommand\contentsname{}
\tableofcontents

\section{Introduction}

In their groundbreaking paper \cite{kl79}, Kazhdan and Lusztig developed a new approach to the representation theory of Hecke algebras. This gave birth to a whole area called Kazhdan-Lusztig theory. Of particular importance in this theory are the objects called cells. Their definition can be summarized as follows. Each Hecke algebra is associated with a Coxeter group $W$. Kazhdan and Lusztig define a pre-order $\leq_L$ on elements of $W$. Some pairs $v, w$ of elements of $W$ satisfy both $v \leq_L w$ and
$w \leq_L v$, in which case they are said to be left-equivalent, denoted $v \sim_L w$. Similarly one can define right equivalence $\sim_R$. The respective equivalence classes are called the left cells and the right cells.

In (finite) type $A$, i.e. when $W$ is the symmetric group, the Kazhdan-Lusztig
cell structure is understood in terms of so-called Robinson-Schensted correspondence. This is a bijective correspondence between elements of the symmetric group and pairs $(P, Q)$ of standard Young tableaux of the same shape. It is well known \cite{kl79, bv82, gm88, ari00} that two permutations lie in the same left (resp. right) cell if and only if they have the same recording tableau $Q$ (resp. insertion tableau $P$).

In affine type $A$, i.e. when $W$ is an affine symmetric group, Chmutov,
Pylyavskyy, and Yudovina \cite{cpy18} constructed a bijection $W \rightarrow \Omega$, where $\Omega$ is the set of triples $(P, Q, \rho)$ such that $P$ and $Q$ are tabloids of the same shape and $\rho$ is an integer vector satisfying certain inequalities. This bijection is called the affine matrix-ball construction, abbreviated AMBC. They show that, parallel to the finite type, fixing the tabloid $Q$ gives all affine permutations in a left cell, while fixing the tabloid $P$ gives all affine permutations in a right cell. Further properties of AMBC were developed by Chmutov-Lewis-Pylyavskyy \cite{clp17}, Chmutov-Frieden-Kim-Lewis-Yudovina \cite{cfkly}, and Kim-Pylyavskyy \cite{kp19}. 

In his work on catabolism and Garsia-Procesi modules Blasiak \cite{bla11} developed an alternative algorithm to determine the left cell to that an affine permutation belongs. He called his algorithm sign insertion. However, some of the key properties of the algorithm remained open. In particular it was in fact Blasiak's conjecture that once one fixes a two-sided cell, recording tableaux of sign insertion algorithm are in one to one correspondence with left cells. In this paper we prove this conjecture using AMBC, thus relating the two existing approaches to Kazhdan-Lusztig combinatorics in affine type $A$. 

Structure of Kazhdan-Lusztig cells can be encoded by $W$-graphs, introduced in the original work of Kazhdan and Lusztig \cite{kl79}. The edges of the $W$-graphs fall into two categories: directed and undirected. The undirected edges are the star operations, and are generally easier to understand than the directed edges. In type $A$ they coincide with Knuth moves on permutations, and they are known to connect each Kazhdan-Lusztig cell into a single connected component. In affine type $A$, however, these star operations are not sufficient to connect the whole cell. The connected components in that case were studied in \cite{clp17}. 

The main tool we use in the proof of Blasiak's conjecture are the new operations called partial rotations and inverse partial rotations. We show that, under certain assumption, they preserve both the left and the right cell of an affine permutation, and thus the only piece of data they change is the weight $\rho$ in terms of AMBC. We also prove that (inverse) partial rotations turn the intersection of a left cell and a right cell into a single connected component. We believe that this makes partial rotations to have independent interest for future study of Kazhdan-Lusztig combinatorics in affine type $A$. 

Lastly, we discuss Blasiak's sign insertion algorithm with the standardization map of Lascoux and Sch\"utzenberger \cite{ls81, las91}. The latter is a function which maps a semistandard Young tableau to another tableau of the same shape but of different content. This function is known to preserve the cocharge statistic of semistandard Young tableaux. More precisely, the set of semistandard Young tableaux of fixed content is equipped with a poset structure whose grading is given by cocharge statistic. Then the function of Lascoux and Sch\"utzenberger yields an graded poset embedding between such posets of different content.

In this paper, we describe how Blasiak's sign insertion can be understood in terms of the standardization map of Lascoux and Sch\"utzenberger. The validity of Blasiak's conjecture states that with each recording tableau of sign insertion algorithm one may associate corresponding row-standard Young tableaux which parametrize corresponding left cells. As these cells are in different two-sided cells, these tableaux should have different shapes. Then our theorem states that the images of (the reading words of) such tableaux under RSK correspondence are closely related in terms of the aforementioned standardization map. (See Theorem \ref{thm:lsstd} for the exact statement.)

The paper proceeds as follows. In Section \ref{sec:not} we introduce notation we use when working with affine symmetric groups. In Section \ref{sec:cells} we remind the reader the key properties of Kazhdan-Lusztig cells and AMBC. In Section \ref{sec:main} we describe Blasiak's sign insertion and state our main theorem. In Section \ref{sec:PPR} we introduce partial rotations, and show how they interact with AMBC and sign insertion. In Section \ref{sec:proof} we prove Blasiak's conjecture. We also prove that partial rotations  connect elements in the intersection of a left cell and a right cell. Finally in Section \ref{sec:LS} we show how Blasiak's sign insertion is related to the standardization map of Lascoux and Sch\"utzenberger.

\section{Notations and Definitions} \label{sec:not}

For a set $A$, we write $\#A$ to be its cardinal. For $a, b \in \bZ$, we set $[a,b] = \{x \in \bZ \mid a\leq x \leq b\}$. A word is a finite sequence of integers, e.g. $w=(w_1, w_2, \ldots, w_k)$ for $k \in \bN=\{0,1,2,\ldots\}$ and $w_1, \ldots, w_k \in \bZ$. For such a word we also write $w=w_1\cdots w_k$. When $k=0$, we also write $w=\emptyset$. For two words $x$ and $y$, we write $x\con y$ to be their concatenation.

We say that $\alpha$ is a composition of $n$, denoted $\alpha \models n$, if $\alpha=(a_1, \ldots, a_k)$ for some $k, a_1, \ldots, a_k \in \bN=\{0, 1, 2, \ldots\}$ where $\sum_{i=1}^k a_i = n$. Here $k$ is said to be the length of $\alpha$, also denoted by $l(\alpha)$. Unless otherwise specified, the length of $\alpha$ is the smallest integer such that $a_{l(\alpha)}\neq 0$. However, sometimes we allow it to have a zero tail at the end. A composition $\alpha =(a_1, \ldots, a_k) \models n$ is called a partition of $n$, denoted $\alpha \vdash n$, if $a_1\geq \cdots \geq a_k>0$. For a composition $\alpha =(a_1, \ldots, a_k)\models n$, we set $\alpha^{rev} = \{a_{l(\alpha)}, \ldots, a_1\}$, which is also a composition of $n$.

A tabloid is a sequence $T=(T_1, \ldots, T_k)$ where each $T_i$ is a sequence of positive integers. We also use its Young diagram to depict $T$ where the $i$-th row is equal to $T_i$. We define its shape $\sh(T)$ to be the composition given by the lengths of parts of $T$. When $\sh(T)$ is a partition we say that $T$ is a tableau. We define the length of the tabloid $T$ to be that of its shape, usually denoted by $l(T)$. For a tabloid $T$, its reading word, denoted $\rw(T)$, is defined to be the concatenation of its rows from bottom to top. For example, if $T=((3,6,7,9),(4,8,10),(1,5),(2))$ then $\rw(T) = (2,1,5,4,8,10,3,6,7,9)$. The content of a tabloid $T$ is a composition, say $\alpha=(a_1, a_2, \ldots)$, where each $a_i$ is the number of letters in $T$ equal to $i$.

For a composition $\alpha=(a_1, \ldots, a_k) \models n$, we write $\RSYT(\alpha)$ to be the set of row-standard Young tabloid of shape $\alpha$, i.e. the set of tabloids $T=(T_1, \ldots, T_k)$ where each element of $[1, n]$ appears exactly once in $T$ and each $T_i$ is an increasing sequence of length $a_i$. Also we let $\RSYT(n) = \bigsqcup_{\lambda \vdash n} \RSYT(\lambda)$, the set of row-standard Young tableaux of size $n$ (only of partition shape). For a composition $\alpha=(\alpha_1, \alpha_2, \ldots) \models n$ and a partition $\lambda \vdash n$, we denote by $\SSYT(\lambda, \alpha)$ the set of semistandard Young tableaux of shape $\lambda$ and content $\alpha$. We also set $\SSYT(n, \alpha) = \bigsqcup_{\lambda \vdash n} \SSYT(\lambda, \alpha)$. When $\alpha=(1^n)$ we write $\SYT(\lambda)$ and $\SYT(n)$ instead of  $\SSYT(\lambda, (1^n))$ and $\SSYT(n, (1^n))$, called the set of standard Young tableaux.

For $n \in \bZ_{>0}$, we set $\extSn$ to be the extended affine symmetric group defined by
\[\extSn = \{ w \colon \bZ \rightarrow \bZ \mid w \textnormal{ is bijective, } w(i+n)=w(i)+n \textnormal{ for } i\in \bZ\}.\]
For $ w\in \extSn$, its window notation is defined to be $[w(1), \ldots, w(n)]$. We often identify $w$ with its window notation. We define the affine symmetric group $\affSn$ to be
\[\affSn = \{ w \in \extSn \mid w(1)+\cdots+w(n) = n(n+1)/2\},\]
which is a subgroup of $\extSn$. We set $s_i = [1, \ldots, i-1, i+1, i, i+2, \ldots, n]$ for $i \in [1,n-1]$ and $s_0=s_n = [0, 2, \ldots, n-1, n+1]$. Then $(\affSn, \{s_1, \ldots, s_{n-1}, s_n=s_0\})$ is an affine Weyl group of type $A$. Moreover, if we set $\shift=[2,3, \ldots, n+1] \in \extSn$ then $\br{\shift}$ is an infinite cyclic group and $\extSn = \br{\shift} \ltimes \affSn$. Here $\shift$ is called the shift element.

For $n \in \bZ_{>0}$, we say that $w$ is a (affine) partial permutation if there exist $X \subset [1,n]$ such that $w$ is an injection from $X+n\bZ$ to $\bZ$ satisfying $w(i+n) = w(i)+n$ for any $i \in X+n\bZ$. Note that $X=[1,n]$ if and only if $w \in \extSn$. We identify $w$ with its window notation $[w(1), \ldots, w(n)]$ where we adopt the convention that  $w(i) = \emptyset$ whenever $i \not \in X$. We often identify a partial permutation $w: X+n\bZ \rightarrow \bZ$ for $X \subset [1,n]$ with its graph $\{(x, w(x)) \mid x \in X+n\bZ\}$. Thus for example, for two partial permutations $u:X+n\bZ \rightarrow \bZ$ and $v: X'+n\bZ \rightarrow \bZ$, we write $u \subset v$ if $X \subset X'$ and $u(x) = v(x)$ for $x \in X$. In addition, we consider the graphs to be placed on the $xy$-plane such that the $x$-axis directs to the south and the $y$-axis directs to the east, i.e. it is obtained from the $xy$-plane with the conventional direction by rotating $90^\circ$ clockwise.

For $n \in \bZ_{>0}$, let $\rot: \bZ \rightarrow \bZ$ be an involution defined by $\rot(x) = n+1-x$. This also induces an involution on the set of partial permutations defined by $w\mapsto \rot \circ w\circ \rot$, which we again denote by $\rot$. Pictorially, it corresponds to rotating the diagram of $w$ by 180 degrees that preserves the square $[1,n] \times [1,n]$. 

For a sequence $A= (A_1, \ldots, A_r)$ and a set $X=\{x_1, \ldots, x_s\}\subset [1,r]$ so that $x_1<x_2<\cdots<x_s$, we set $A_X=(A_{x_1}, \ldots, A_{x_s})$. If $X=[1,s]$ (resp. $X=[s,r]$), we also write $A_{\leq s}$ (resp. $A_{\geq s}$) instead of $A_X$. We set $A\su{X} =(\emptyset, A_{x_1}, \emptyset, \ldots, \emptyset,A_{x_s},\emptyset)$ (i.e replacing $A_i$ with $\emptyset$ if $i \not\in X$). When $w \in \extSn$ and $X \subset [1,n]$ we set $w\su{X}= (w(1), \ldots, w(n))\su{X}$, i.e. we identify $w$ with its window notation and discard $w(i)$ for $i \not\in X$. Note that this becomes a window notation of some partial permutation; we abuse notation and also denote it by $w\su{X}$.

\section{Cells, star operations, and AMBC}\label{sec:cells}
Kazhdan and Lusztig \cite{kl79} introduced (two-sided, left, and right) \emph{cells} for each Coxeter group and defined (right and left) star operations. (The notion of cells is extended to extended affine Weyl groups by Lusztig \cite{lus89:cell}.) Here we focus on the case of (extended) affine symmetric groups in terms of the affine matrix-ball construction.

For a row-standard Young tableau $T$ and $i\in [1, l(T)-1]$ suppose that $T_i=(a_1, \ldots, a_s)$ and $T_{i+1} = (b_1, \ldots, b_t)$. Then  we define $\lch_i(T)$, called the local charge in row $i$ of $T$, to be the smallest $d\in \bN$ such that $a_{l-d}<b_{l}$ for $l\in [d+1, t]$.\footnote{One can easily show that this definition coincides with \cite[Definition 5.3]{clp17}.} Pictorially, this corresponds to how much one needs to shift $T_i$ to the right so that $(T_i, T_{i+1})$ becomes standard. For example, if $T_i = (3,5,7,8)$ and $T_{i+1} = (1,2,4,6)$ then we have $\lch_i(T) = 2$ as depicted below.
\[\ytableaushort{3578,1246} \quad \Rightarrow\quad \ytableaushort{{\none}{\none}3578,1246}
\]

For $P, Q \in \RSYT(\lambda)$ where $\lambda=(\lambda_1, \ldots, \lambda_l)$ is a partition, we define $\offset_{P,Q}=(s_1, \ldots, s_l) \in \bZ^l$, called the symmetrized offset constant of $(P,Q)$, to be as follows.
\[s_i = 
\left\{\begin{aligned}
&0& \textnormal{ if } i=1 \textnormal{ or } \lambda_{i-1}>\lambda_{i},
\\&s_{i-1}+\lch_{i-1}(P)-\lch_{i-1}(Q) & \textnormal{ otherwise}.
\end{aligned}\right.
\]
In other words, we have $s_i-s_{i-1} = \lch_{i-1}(P)-\lch_{i-1}(Q)$ whenever $\lambda_{i-1}=\lambda_i$. 
We say that $\vv{\rho}=(\rho_1, \ldots, \rho_l)$ is dominant with respect to $(P,Q)$ if $\rho_{i-1}-s_{i-1}\leq \rho_{i}-s_{i}$ whenever $\lambda_{i-1}=\lambda_i$.

The affine matrix-ball construction (abbreviated AMBC), defined in \cite{cpy18}, is a systematic method to understand the combinatorial properties of $\extSn$ and yields two functions 
\begin{gather*}
\Phi: \extSn \rightarrow \bigsqcup_{\lambda \vdash n} \RSYT(\lambda) \times \RSYT(\lambda) \times \bZ^{l(\lambda)} \ \ \textnormal{ and } \ \ 
\Psi: \bigsqcup_{\lambda \vdash n} \RSYT(\lambda) \times \RSYT(\lambda) \times \bZ^{l(\lambda)} \rightarrow \extSn.
\end{gather*}
Here, we have $\im \Phi = \{ (P, Q, \vv{\rho})  \mid \vv{\rho} \textnormal{ is dominant with respect to } (P,Q)\}$ and  $\im \Psi = \extSn$. Moreover, $\Psi|_{\im \Phi}$ and $\Phi$ are inverse to each other. (Their constructions are explained in \ref{sec:AMBC} in more detail.)

The two-sided cells of $\extSn$ are parametrized by partitions of $n$ and we denote such cells by $\tsc_\lambda$ for $\lambda \vdash n$. The left/right cells contained in $\tsc_\lambda$ are parametrized by row-standard Young tableaux of shape $\lambda$, and we write $\lc_T$ (resp. $(\lc_T)^{-1}$) for $T\in \RSYT(\lambda)$ for such a left (resp. right) cell. Here, we have $w \in \lc_T$ if and only if $\Phi(w) = (P, T, \vv{\rho})$ for some $P$ and $\vv{\rho}$, and similarly  $w \in (\lc_T)^{-1}$ if and only if $\Phi(w) = (T, Q, \vv{\rho})$ for some $Q$ and $\vv{\rho}$. Also we have $\tsc_\lambda = \bigsqcup_{T \in \RSYT(\lambda)} \lc_T= \bigsqcup_{T \in \RSYT(\lambda)} (\lc_T)^{-1}$.

For $w\in \extSn$ and $i \in [1,n]$ such that either $w(i)<w(i+2)<w(i+1)$, $w(i+1)<w(i+2)<w(i)$, $w(i)<w(i-1)<w(i+1)$, or $w(i+1)<w(i-1)<w(i)$, we define the right star operation $w \mapsto w^*$ for $* \sim i$ to be the one obtained from $w$ by swapping $w(i+kn)$ and $w(i+1+kn)$ for each $k \in \bZ$. Similarly, we define the left star operation $w \mapsto {}^*w$ for $* \sim i$ to be the composition $w \mapsto w^{-1} \mapsto (w^{-1})^* \mapsto ((w^{-1})^*)^{-1}$ when the corresponding right star operation $w^{-1} \mapsto (w^{-1})^*$ for $* \sim i$ is well-defined. Then \cite[Proposition 5.10]{kp19} states that when $\Phi(w) = (P, Q, \vv{\rho})=(P, Q,\offset_{P, Q}+ \vv{\crho})$ we have
\begin{align*}
\Phi({}^*w) &=\left\{
\begin{aligned}
&(P^*, Q, \offset_{P^*,Q}+\vv{\crho}) & \textnormal{ if } *\not\sim n,
\\&(P^*, Q, \offset_{P^*,Q}+\vv{\crho}-\vv{\delta}(P,1)+\vv{\delta}(P,n)) & \textnormal{ if } *\sim n.
\end{aligned}\right.
\\\Phi(w^*) &=\left\{
\begin{aligned}
&(P, Q^*, \offset_{P,Q^*}+\vv{\crho}) & \textnormal{ if } *\not\sim n,
\\&(P, Q^*, \offset_{P,Q^*}+\vv{\crho}+\vv{\delta}(Q,1)-\vv{\delta}(Q,n)) & \textnormal{ if } *\sim n.
\end{aligned}\right.
\end{align*}
Here, for $T \in \RSYT(\lambda)$ and $s \in [1,n]$ such that $s \in T_a$, we define $\vv{\delta}(T,s)=(\delta_1, \ldots, \delta_l)$ to be
\[ \delta_i = \left\{
\begin{aligned} &1 & \textnormal{ if } \lambda_{a-1}>\lambda_a=\lambda_i   \textnormal{ (here we set $\lambda_{0}=\infty$), and}
\\ & 0 & \textnormal{ otherwise}.
\end{aligned}
\right.
\]
(In particular, $\delta_i=\delta_j$ whenever $\lambda_i=\lambda_j$.)

Multiplication by $\shift=[2,3,\ldots, n+1] \in \extSn$ preserves each two-sided cell but permutes left/right cells. More precisely, when $\Phi(w)=(P, Q, \vv{\rho}) = (P, Q, \offset_{P,Q}+\vv{\crho})$ we have (see \cite[Proposition 5.5]{kp19})
\begin{align*}
\Phi(\shift\cdot w) &= (\shift(P), Q, \offset_{\shift(P),Q}+\vv{\crho}+\vv{\delta}(P,n)), \textnormal{ and}
\\\Phi(w\cdot\shift^{-1}) &= (P, \shift(Q), \offset_{P,\shift(Q)}+\vv{\crho}-\vv{\delta}(Q,n)).
\end{align*}
Here, $\shift(T)$ is a tableau obtained from $T$ by changing each entry $i$ with $i+1$ (and reordering entries in each row if necessary).

\begin{example} Suppose that $w=[1, 6, 8, 14, 17, 5, 0, 19, 3, 22]\in \extS{10}$. Then we have
\begin{align*}
{}^*w&=[1, 6, 8, 14, 17, 5, 0, 19, 2, 23] \textnormal{ for } *\sim 2,
\\w^*&=[12, 6, 8, 14, 17, 5, 0, 19, 3, 11] \textnormal{ for } *\sim 10,
\\\shift w&=[2, 7, 9, 15, 18, 6, 1, 20, 4, 23], 
\\w\shift^{-1}&=[12, 1, 6, 8, 14, 17, 5, 0, 19, 3].
\end{align*}
In each case we have (all the symmetrized offset constants are zero as the parts of $(4,3,2,1)$ are pairwise different)
\begin{align*}
\Phi(w) &= (((3,5,8,10), (1,4,9), (2,6), (7)), ((4,5,8,10),(1,2,3),(6,9),(7)), (4,0,0,0)),
\\\Phi({}^*w)&=(((2,5,8,10), (1,4,9), (3,6), (7)), ((4,5,8,10),(1,2,3),(6,9),(7)), (4,0,0,0)),
\\\Phi(w^*)&=(((3,5,8,10), (1,4,9), (2,6), (7)), ((1,4,5,8),(2,3,10),(6,9),(7)), (3,1,0,0)),
\\\Phi(\shift w) &= (((1,4,6,9), (2,5,10), (3,7), (8)), ((4,5,8,10),(1,2,3),(6,9),(7)), (5,0,0,0)),
\\\Phi(w\shift^{-1}) &=  (((3,5,8,10), (1,4,9), (2,6), (7)), ((1,5,6,9),(2,3,4),(7,10),(8)), (3,0,0,0))
\end{align*}
which are as expected from the observation above.
\end{example}


%

\section{Sign insertion and the main theorem} \label{sec:main}
Here we recall the sign insertion algorithm defined in \cite{bla11} and state the first main result of this paper (Theorem \ref{thm:main}). Suppose that we have words $\fP, \fQ, w$ such that $w\neq \emptyset$ and $\fP\con w$ is a window notation of some element in $\extSn$. Also we assume that $\fP$ and $\fQ$ are both increasing and of the same length.  Let us write $w=(a) \con w''$ where $a\in \bZ$ is the first entry of $w$ and $w''$ is the remaining part. For $i \in \bZ_{>0}$ bigger than any entry in $\fQ$, we define $\digamma(i-1, \fP, \fQ, w) = (i, \fP', \fQ', w')$ as follows.
\begin{enumerate}[label=\textbullet]
\item If $a$ is bigger than any entry of $\fP$, then we set $\fP'=\fP\con (a)$, $\fQ'=\fQ\con(i)$, and $w'=w''$.
\item Otherwise, let $b$ be the smallest element in $\fP$ bigger than $a$. Then we set $\fP'$ to be $\fP$ after replacing $b$ with $a$. Also we set $\fQ'=\fQ$ and $w' = w''\con (n+b)$.
\end{enumerate}
For $w \in \extSn$, we start with the quadruple $(0, \emptyset, \emptyset, w)$ and apply $\digamma$ recursively until the last component becomes the empty word, say $(i, \fP, \fQ, \emptyset)$. We write $\sgn_\fP(w) = \fP$ and $\sgn_\fQ(w) = \fQ$. The process $w \mapsto (\sgn_{\fP}(w), \sgn_{\fQ}(w))$ is called the sign insertion.

\begin{example} Let $w=[17,13,4, 20, 9, 24] \in \extS{6}$. Each step of the sign insertion applied to $w$ is described as follows.
\[
\begin{array}{|c|c|c|c|}
\hline
i&\fP & \fQ & w\\
\hline
0&\emptyset & \emptyset & 17,13, 4, 20, 9, 24\\
1&17 &1 & 13, 4, 20, 9, 24\\
2&13 & 1& 4, 20, 9, 24,23\\
3&4 & 1 & 20, 9, 24,23, 19\\
4&4,20 & 1,4 &  9, 24,23, 19\\
5&4,9 & 1,4 &   24,23, 19,26\\
6&4,9,24 & 1,4,6 &   23, 19,26\\
7&4,9,23 & 1,4,6 &   19,26,30\\
8&4,9,19 & 1,4,6 &   26,30,29\\
9&4,9,19,26 & 1,4,6,9 &  29\\
10&4,9,19,26,30 & 1,4,6,9,10 &   29\\
11&4,9,19,26,29 & 1,4,6,9,10 &   36\\
12&4,9,19,26,29,36 & 1,4,6,9,10,12 &   \emptyset\\
\hline
\end{array}
\]
Thus we have $\sgn_\fP(w) = (4,9,19,26,29,36)$ and $\sgn_\fQ(w) = (1,4,6,9,10,12)$.
\end{example}

\begin{rmk} Note that it is not exactly the same as, but equivalent to, the original definition of Blasiak \cite[Algorithm 9.7]{bla11}. Indeed, in our convention the window notation of $w\in \extSn$ is $[w(1), w(2), \ldots, w(n)]$, whereas in \cite{bla11} it is defined to be $[n+1-w^{-1}(1), n+1-w^{-1}(2), \ldots, n+1-w^{-1}(n)]$. Furthermore, our sign insertion algorithm uses the first element of $w$ in each step, whereas in \cite{bla11} the last element is used. Therefore, the Blasiak's insertion algorithm applied to $w$ is equivalent to our algorithm applied to $\rot(w^{-1})$.
\end{rmk}

Let us discuss some basic properties of sign insertion.
\begin{lem}\label{lem:signleft} Suppose that $\digamma(i-1,\fP, \fQ, w) = (i, \fP', \fQ', w')$. Then either $\fP \con w=\fP' \con w'$ or $\fP' \con w'$ is obtained from $\fP \con w$ by applying a series of right star operations followed by multiplication by $\shift$ on the right. As a result, $w$ and $\sgn_\fP(w)$ are in the same left cell for $w\in \extSn$.
\end{lem}
\begin{proof} Suppose that $\fP \con w\neq\fP' \con w'$. If we let $w= (a)\con w''$, then $\fP'$ is obtained from $\fP$ by the ``bumping process'', i.e. regarding $\fP$ as an one-row tableau and inserting $a$ to it using the usual Robinson-Schensted insertion algorithm. Thus if we let $b$ be the element bumped out from $\fP$, then it is clear that $(b)\con \fP'$ is obtained from $\fP \con (a)$ by a series of right star operations. Now the result follows as $((b)\con\fP'\con w'')\cdot \shift=\fP' \con w'' \con (b+n) = \fP' \con w'$.
\end{proof}

\begin{lem}\label{lem:sgnqpre}Suppose that $w\in \extSn$.
\begin{enumerate} 
\item $\sgn_\fQ(w) = \sgn_\fQ(\shift w)$.
\item If ${}^*w$ is well-defined for some $*\sim i$, then $\sgn_\fQ({}^*w) = \sgn_\fQ(w)$.
\end{enumerate}
\end{lem}
\begin{proof} This is a reformulation of \cite[Proposition 9.9 (h), (i)]{bla11} with respect to our convention. In \cite{bla11} he only considers when $*\not\sim n$ (see the comment right above \cite[Example 9.3]{bla11}), however the argument still applies when $*\sim n$.
\end{proof}

Now we state the first main result of this paper, originally conjectured by J. Blasiak.
\begin{thm} \label{thm:main}
\cite[Conjecture 9.17(b)]{bla11} For $\lambda \vdash n$ and $\fQ$, the set $\{w \in \tsc_\lambda \mid \sgn_\fQ(w)=\fQ\}$ is a single left cell (if nonempty).
\end{thm}
\begin{rmk} Since his $\sgn_Q(w)$ is equal to our $\sgn_\fQ(\rot(w^{-1}))$, his conjecture is translated to that $\{w \in \tsc_\lambda \mid \sgn_\fQ(\rot((w^{-1})^{-1}))=\sgn_\fQ(\rot(w))=\fQ\}$ is a single left cell. However, since $\rot: \extSn \rightarrow \extSn$ sends a left cell to a left cell (see \cite[Theorem 3.1]{cfkly} for more detail), these two statements are equivalent.
\end{rmk}

\section{AMBC, proper partial rotations, and sign insertion}\label{sec:PPR}
Here we first review the construction of AMBC in detail. Then we introduce \emph{proper partial rotations} which play a central role in our paper and discuss how they are related to both AMBC and Blasiak's sign insertion. Eventually we verify the Diamond Lemma (Lemma \ref{lem:diamond}) which is a key step for the proof of the Blasiak's conjecture.

\subsection{Notions for AMBC} \label{sec:AMBC}
Here we recap some definitions from \cite{cpy18} that will be used in this paper. We fix a positive integer $n$.

For a partial permutation $w$, we define its Shi poset to be the poset $P=P_w$ on $\{x \in [1,n] \mid w(x)\neq \emptyset\}$ such that $i <_{P} j$ if either [$i>j$ and $w(i)<w(j)$] or [$w(j)>w(i)+n$]. For such a poset, we define its Greene-Kleitman partition to be $\lambda=(\lambda_1, \lambda_2, \ldots)$ where $\sum_{i=1}^k\lambda_i$ is the maximum of the number of elements in the union of $k$-antichains in $P_w$. Then by \cite{shi91}, if $w\in \extSn$ is an actual permutation then the Greene-Kleitman partition $\lambda$ of $P_w$ parametrizes the two-sided cell that contains $w$, i.e. we have $w \in \tsc_\lambda$. In such a case we define the width of $P_w$ to be the first row of $\lambda$, which is also equal to the maximum length of an antichain in $P_w$.

\begin{example}\label{ex:shi}
Let $w=[8,1,19, 14,16, 2, 25, 13, 10, 27] \in \extS{10}$. Its Shi poset $P_w$ is described in Figure \ref{fig:hasse}. Direct calculation shows that its Greene-Kleitman partition is $(3, 3, 3, 1)$, which means that $w \in \tsc_{(3,3,3,1)}$.
\begin{figure}[!htbp]
\begin{tikzpicture}
	\tikzmath{\x=6;}
	\draw[] (0,0) circle (\x pt) node (2) {2};
	\draw[] (1,0) circle (\x pt) node (6) {6};
	\draw[] (2,0) circle (\x pt) node (9){9};
	\draw[] (0.5,1) circle (\x pt) node (1){1};
	\draw[] (1.5,1) circle (\x pt) node (8){8};
	\draw[] (0.5,2) circle (\x pt) node (4){4};
	\draw[] (1.5,2) circle (\x pt) node (5){5};
	\draw[] (0,3) circle (\x pt) node (3){3};
	\draw[] (1,3) circle (\x pt) node (7){7};
	\draw[] (2,3) circle (\x pt) node (10){10};	
        \draw[-latex] (9) -- (8);
        \draw[-latex] (6) -- (8);
        \draw[-latex] (6) -- (1);
        \draw[-latex] (2) -- (8);
        \draw[-latex] (2) -- (1);
        \draw[-latex] (8) -- (4);
        \draw[-latex] (8) -- (5);
        \draw[-latex] (4) -- (3);
        \draw[-latex] (4) -- (7);
        \draw[-latex] (4) -- (10);
        \draw[-latex] (5) -- (3);
        \draw[-latex] (5) -- (10);
        \draw[-latex] (1) -- (3);
        \draw[-latex] (1) -- (7);
        \draw[-latex] (1) -- (10);
\end{tikzpicture}
\caption{Hasse diagram of $P_{[8,1,19, 14,16, 2, 25, 13, 10, 27]}$}
\label{fig:hasse}
\end{figure}
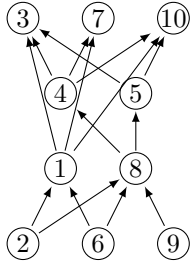
\end{example}

We say that a partial permutation $S: X +n\bZ \rightarrow \bZ$ for some $X\subset [1,n]$ is a stream if $S$ is an increasing function. In such a case its density is defined to be $\#X$. When there exists another partial permutation $w$ such that $S \subset w$, we say that $S$ is a channel of $w$ if the density of $S$ is equal to the width of $P_w$. Indeed, a stream $S \subset w$ is a channel if and only if its density is maximum among the streams contained in $w$. For a stream $S:X +n\bZ \rightarrow \bZ$ its altitude $a(S)$ is defined to be $\sum_{x\in X} (\ceil{w(x)/n}-1)$ where $\ceil{t}$ is the smallest integer not smaller than $t$.

For $(x,y),(x',y') \in \bZ^2$ we write $(x,y) \leq_{SW} (x',y')$ if $x\geq x'$ and $y\leq y'$ and say that $(x,y)$ is southwest of $(x',y')$ or equivalently $(x',y')$ is northeast of $(x,y)$. Similarly, we write $(x,y)\leq_{NW}(x',y')$ if $x\leq x'$ and $y\leq y'$ and say that $(x,y)$ is northwest of $(x',y')$ or equivalently $(x',y')$ is southeast of $(x,y)$. (Note that this is compatible with our convention of the $xy$-plane.) For two channels $C, C' \subset w$ we say that $C$ is southwest of $C'$, denoted $C \leq_{SW} C'$, if for any $b\in C$ there exists $b'\in C'$ such that $b\leq_{SW}b'$. If neither $C\leq_{SW} C'$ nor $C\geq_{SW} C'$ then we write $C \not\sim_{SW} C'$. By \cite[Proposition 3.13]{cpy18}, for any partial permutation $w$ there always exist channels $C, C' \subset w$ such that $C\leq_{SW} C'' \leq_{SW} C'$ for any channel $C'' \subset w$. We call $C$ (resp. $C'$) the southwest (resp. northeast) channel of $w$.

For a partial permutation $w$, a sequence $((x_1, w(x_1)), \ldots, (x_k, w(x_k)))$ is called a path in $w$ if $x_1<\cdots <x_k$ and $w(x_1)<\cdots <w(x_k)$, i.e. it is a chain with respect to the southeast direction. Or equivalently, if $(b_1, \ldots, b_k)$ then we have $b_1 \leq_{NW} \cdots \leq_{NW}b_k$. \footnote{This is called a reverse path in \cite{cpy18} and a path is defined to be of the opposite direction. Here we stick to our definition in order to avoid confusion to which way a path is directed.} A numbering $d:w \rightarrow \bZ$ is called monotone if for any path $(b_1, b_2, \ldots, b_k)$ we have $d(b_1)<d(b_2)<\cdots<d(b_k)$.

For a stream $S$ we say that $\tilde{d}: S \rightarrow \bZ$ is a proper numbering if $\tilde{d}(x, S(x))<\tilde{d}(y, S(y))$ whenever $x<y$ and $\tilde{d}(x+n, S(x+n))-\tilde{d}(x, S(x))$ equals the density of $S$ for any $x$. Note that a proper numbering is unique up to shift. When $C$ is a channel of $w$ and $\tilde{d}:C \rightarrow \bZ$ is a proper numbering, we define the channel numbering $d^C_w=d^C: w \rightarrow \bZ$ to be 
\[d^C(b) = \max \{\tilde{d}(b')+k \mid \textnormal{there exists a path } (b'=b_0, b_1, \ldots, b_k=b) \textnormal{ in } w \textnormal{ where } b' \in C\}.\]
Again $d^C$ is uniquely determined up to shift. When $C$ is the southwest channel of $w$ we also write $d=d^C$ (or $d_w=d_w^C$), called the southwest channel numbering of $w$.

For two channels $C, C' \subset w$, we fix shifts of $d^{C}$ and $d^{C'}$ so that they coincide on some/any element in $C$. Then the distance between $C$ and $C'$ is defined to be $h(C, C') = |d^{C'}(b)-d^{C}(b)|$ for some/any $b \in C'$. Then this is well-defined and satisfies that $h(C, C')+h(C', C'') \geq h(C, C'')$ for any channels $C, C', C'' \subset w$. Furthermore, if $C\not\sim_{SW} C'$ then $h(C, C')=0$ by \cite[Lemma 11.12]{cpy18}. (The converse is not true in general.) We say that $R$ is a river of $w$ if it is of the form $R=\{C'\subset w \mid C' \textnormal{ is a channel of } w, h(C, C')=0\}$ for some channel $C \subset w$.

\begin{example}\label{ex:chnum} Let $w=[8,1,19, 14,16, 2, 25, 13, 10, 27] \in \extS{10}$ as in Example \ref{ex:shi}. There are three channels of $w$ as follows:
\begin{align*}
C_1 = [\emptyset, 1, \emptyset ,\emptyset,\emptyset,2,\emptyset,\emptyset,10,\emptyset], 
\quad C_2 =[8, \emptyset, \emptyset, 14, 16, \emptyset, \emptyset,\emptyset, \emptyset, \emptyset],
\quad C_3 =[\emptyset, \emptyset, 19, \emptyset, \emptyset, \emptyset, 25, \emptyset, \emptyset, 27].
\end{align*}
Here $C_1$ is the southwest channel of $w$ and $C_3$ is the northeast channel of $w$. Let us fix the proper numbering $\tilde{d}$ of $C_1$ such that $\tilde{d}(2, 1) = 0$. Then the southwest channel numbering $d=d^{C_1}: w \rightarrow \bZ$ is defined to be
\begin{gather*}
(1,8), (2,1) \mapsto 0, \qquad (4, 14), (6,2) \mapsto 1, \qquad (3,19), (5,16), (8,13), (9,10) \mapsto 2,
\\ (7,25) \mapsto 3, \qquad (10, 27) \mapsto 4.
\end{gather*}
Similarly, $d^{C_2}:w \rightarrow \bZ$ is defined to be
\begin{gather*}
(2, 1) \mapsto 0, \qquad
(1, 8), (6, 2) \mapsto 1,\qquad
(4, 14), (8, 13),(9, 10) \mapsto 2
\\ (5, 16), (3, 19) \mapsto 3,\qquad
(7, 25) \mapsto 4,\qquad
(10, 27) \mapsto 5,
\end{gather*}
and $d^{C_3}:w \rightarrow \bZ$ is defined to be
\begin{gather*}
(2, 1) \mapsto 0, \qquad
(1, 8), (6, 2) \mapsto 1,\qquad
(4, 14), (8, 13),(9, 10) \mapsto 2
\\ (5, 16) \mapsto 3,\qquad
(3, 19) \mapsto 4,\qquad
  (7, 25) \mapsto 5, \qquad 
  (10, 27) \mapsto 6.
\end{gather*}
Therefore we see that $h(C_1, C_2) = 1$, $h(C_2, C_3) = 1$, and $h(C_1, C_3)=2$.
\end{example}

For a partial permutation $w$ and its channel $C \subset w$, we define $\fw_C(w)$ and $\st_C(w)$ as follows. Let $d=d^C_w: w \rightarrow \bZ$ be the channel numbering defined above. For each $m\in \bZ$, we take $x_1, \ldots, x_k\in \bZ$ such that $x_1>\cdots>x_k$ and $d(x_i, w(x_i))=m$ for $i \in [1,k]$. Then we have $w(x_1)<\cdots <w(x_k)$. (It follows from the monotone property of $d$.) Set $Z_m=Z_m'\sqcup Z_m'' \sqcup Z_m'''$ where
\[
Z_m'= \{ (x_i, w(x_i)) \mid i \in [1,k]\}, \qquad Z_m'' = \{ (x_i, w(x_{i+1})) \mid i \in [1,k-1]\}, \qquad Z_m''' = \{ (x_k, w(x_{1}))\}.
\]
The set $Z_m$ is called the zigzag labeled $m$. Then $w \cap Z_m = Z_m'$. We set $\fw_C(w) = \bigsqcup_{m\in \bZ}Z_m''$ and $\st(w) = \bigsqcup_{m\in \bZ}Z_m'''$. Then $\fw_C(w)$ is a partial permutation and $\st_C(w)$ is a stream whose density is equal to the width of $P_w$. When $C$ is the southwest channel of $w$ we simply write $\fw(w)$ and $\st(w)$ instead of $\fw_C(w)$ and $\st_C(w)$.

\begin{example} Let $w=[8,1,19, 14,16, 2, 25, 13, 10, 27] \in \extS{10}$ as in Example \ref{ex:shi} and \ref{ex:chnum}. Then with respect to $d=d^{C_1}$ we have
\begin{align*}
Z_0'&=\{(2,1),(1,8),  (-3, 15) \}, &&Z_0'' =\{(2,8), (1, 15)\}, &&Z_0''' = \{(-3, 1)\},
\\Z_1'&=\{(6,2), (4,14), (0, 17) \}, &&Z_1'' =\{(6,14), (4,17)\}, &&Z_1''' = \{(0, 2)\},
\\Z_2'&=\{(9,10), (8,13), (5, 16), (3,19) \}, &&Z_2'' =\{(9,13), (8,16), (5,19)\}, &&Z_2''' = \{(3, 10)\}.
\end{align*}
Also we have $\fw(w) = (Z_0''\sqcup Z_1''\sqcup Z_2'') +(n,n)\bZ$ and $\st(w) = (Z_0'''\sqcup Z_1'''\sqcup Z_2''') +(n,n)\bZ$.
\end{example}

The map $\Phi: \extSn \rightarrow \bigsqcup_{\lambda \vdash n} \RSYT(\lambda) \times \RSYT(\lambda) \times \bZ^{l(\lambda)}$ is defined as follows. Starting with $w_0=w \in \extSn$, we successively calculate $w_{i+1}=\fw(w_{i})$ and $S_{i+1}=\st(w_{i})$ until we have an empty partial permutation. For each $S_{i}$ we set $P_i, Q_i \subset [1,n]$ to be such that $S_i : P_i +n\bZ \rightarrow Q_i+n\bZ$ is a bijection. Then we have $\Phi(w) = (P, Q, \vv{\rho})$ where $P=(P_1, P_2,\ldots)$, $Q=(Q_1, Q_2,\ldots)$, and $\vv{\rho}=(a(S_1), a(S_2), \ldots)$. Moreover, $\vv{\rho}$ is always dominant with respect to $(P,Q)$.

A stream $S$ is said to be compatible with a partial permutation $w$ if $S\cap w =\emptyset$, $S \cup w$ is still a partial permutation, and the density of $S$ is not (strictly) smaller than the width of $P_w$. In such a case we define the backward numbering $d^{\bk,S}_w=d^{\bk,S}: w \rightarrow \bZ$ as follows. First we fix a proper numbering $\tilde{d}:S \rightarrow \bZ$, and for $(x, w(x)) \in w$ we let $d(x, w(x))=\max\{\tilde{d}(y, S(y)) \in S \mid y<x, S(y)<w(x)\}$. Now we repeat the following process:
\begin{enumerate}[label=\textbullet]
\item If $d(x, w(x))<d(y, w(y))$ for any $x, y$ such that $x<y$ and $w(x)<w(y)$, then we terminate the process.
\item Otherwise, choose $(x, w(x))$ such that:
\begin{enumerate}[label=-]
\item there exists $y$ such that $d(x,w(x))\geq d(y,w(y))$, $x<y$, and $w(x)<w(y)$, and
\item for any $z$ such that $z<x$ and $w(z)<w(x)$ we have $d(z,w(z))<d(x,w(x))$.
\end{enumerate}
\item For each $i \in \bZ$ we lower the value of $d(x+in, w(x)+in)$ by 1.
\end{enumerate}
After this process is finished we set $d^{\bk, S}_w=d$. This numbering is always well-defined.

For a partial permutation $w$ and a compatible stream $S$, we define $\bk_S(w)$ as follows. Let $\tilde{d}$ be the proper numbering on $S$ and $d=d^{\bk,S}_w$ be the induced backward numbering on $w$. Then for each $m\in \bZ$, we take $y, x_1, \ldots, x_k \in \bZ$ such that $x_1>\cdots >x_k$ and $\tilde{d}(y, S(y)) = d(x_i, w(x_i))=m$ for $i \in [1,k]$. Then we have $w(x_1)<\cdots <w(x_k)$. (It follows from the monotone property of $d$.) Set $Z_m = Z_m'\sqcup Z_m''\sqcup Z_m'''$ where
\begin{gather*}
Z_m'= \{ (x_{i+1}, w(x_i)) \mid i \in [1,k-1]\}\sqcup\{(x_1, S(y))\}\sqcup\{(y, w(x_k))\},
\\ Z_m'' = \{ (x_i, w(x_{i})) \mid i \in [1,k]\}, \qquad Z_m''' = \{ (y, S(y))\}.
\end{gather*}
The set $Z_m$ is called the zigzag labeled $m$. Then $w \cap Z_m = Z_m''$ and $S \cap Z_m = Z_m'''$. We set $\bk_S(w) = \bigsqcup_{m \in \bZ} Z_m'$. Then $\bk_S(w)$ is a partial permutation the width of whose Shi poset is equal to the density of $S$.

The map $\Psi: \bigsqcup_{\lambda \vdash n} \RSYT(\lambda) \times \RSYT(\lambda) \times \bZ^{l(\lambda)} \rightarrow \extSn$ is defined as follows. For $P=(P_1, \ldots, P_l)$, $Q=(Q_1, \ldots, Q_l)$, and $\vv{\rho} = (\rho_1, \ldots, \rho_l)$, we define $S_i$ to be the unique stream which maps $P_i+n\bZ$ to $Q_i+n\bZ$ and which satisfies $a(S_i)=\rho_i$. Now starting with the empty partial permutation $w_l$ we successively define $w_{i-1}=\bk_{S_i}(w_i)$ for $i \in [1,l]$. This process is well-defined and we set $\Psi(P, Q, \vv{\rho}) = w_0$. (Here the dominance of $\vv{\rho}$ is not required.)

\subsection{Channels and rivers} \label{sec:chan}
Let us describe some properties of channels and rivers. We start with observing the following lemma which strengthens \cite[Proposition 3.17]{cpy18}.
\begin{lem} \label{lem:addh} Suppose that $C_1, C_2, C_3$ are three (not necessarily disjoint) channels of a partial permutation $w$. Then we have $h(C_1, C_3)\leq h(C_1, C_2)+h(C_2, C_3)$. Furthermore, if $h(C_1, C_2)=0$, $h(C_2, C_3)=0$, or $C_1 \leq_{SW} C_2 \leq_{SW}C_3$, then we have $h(C_1, C_2)+h(C_2,C_3) = h(C_1, C_3)$.
\end{lem}
\begin{proof} It is indeed shown in the proof of \cite[Proposition 3.17]{cpy18}.
\end{proof}

For a partial permutation of $w$ whose Shi poset is of width $r$, we successively choose $C_i$ to be the southwest channel of $w-\sqcup_{j=1}^{i-1} C_j$ until we fail to find a stream of density $r$ and let $\{C_1, \ldots, C_m\}$ be its result. (This is always possible due to \cite[Proposition 3.13]{cpy18}.) Then $(C_1, \ldots, C_m)$ is a disjoint collection of channels of $w$ where $C_1 \leq_{SW} \cdots \leq_{SW} C_m$. By \cite[Definition 3.7]{cpy18} and the argument thereafter, we see that $m$ equals the multiplicity of $r$ (the width of $P_w$) in the Greene-Kleitman partition of $P_w$. By Lemma \ref{lem:addh}, in our sitution there exist $0=m_0,m_1, m_2, \ldots, m_l=m$ where $0<m_1<m_2<\ldots<m_{l-1}<m$ such that $h(C_i, C_j)=0$ if and only if $i,j \in [1+m_{k-1}, m_{k}]$ for some $k \in [0,l-1]$. In other words, there are $l$ rivers $R_1, \ldots, R_l \subset w$ such that $C_i \in R_k$ for $k \in [1,l]$ and $i\in [1+m_{k-1},m_{k}]$. Then, we may replace $C_{m_k}$ for each $k \in [1,l]$ with the northeast channel of $R_k$ without violating the condition $C_1\leq_{SW}\cdots \leq_{SW} C_m$ and the pairwise disjoint property.

\begin{defn} For a partial permutation $w$, we set $\fC_w=(C_1, \ldots, C_m)$ to be the sequence of the channels of $w$ such that (1) $m$ is the multiplicity of $\lambda_1$ in the Greene-Kleitman partition $\lambda$ of $P_w$, (2) $C_1 \leq_{SW} C_2 \leq_{SW} \cdots \leq_{SW} C_m$, (3) $C_i$'s are pairwise disjoint, and (4) this sequence contains all the northeast and southwest channels of rivers of $w$.
\end{defn}

\begin{example}\label{ex:fCw} Suppose that $w=[6,1,18,3,19,24,12,15,17,10]$. Then there are four channels
\begin{align*}
C_1=&\ [\emptyset, 1,\emptyset, 3, \emptyset,\emptyset,\emptyset,\emptyset,\emptyset,10],
\\C_2'=&\ [6,\emptyset,\emptyset,\emptyset,\emptyset,\emptyset,12,15,\emptyset,\emptyset],
\\C_2=&\ [\emptyset,\emptyset,\emptyset,\emptyset,\emptyset,\emptyset,12,15,17,\emptyset],
\\C_3=&\ [\emptyset,\emptyset,18,\emptyset,19,24,\emptyset,\emptyset,\emptyset,\emptyset].
\end{align*}
Here $\{C_1, C_2, C_2'\}$ and $\{C_3\}$ are rivers of $w$. We take $C_2$ instead of $C_2'$ as $C_2$ is the northeast channel in its river. Thus we have $\fC_w = (C_1, C_2, C_3)$.
\end{example}

Suppose that $m=\#\fC_w \geq 2$. Then by \cite[Corollary 14.11]{cpy18}, if we set $\fC_{\fw(w)}=(D_1, \ldots, D_{m-1})$ then each $D_i$ is between $C_i$ and $C_{i+1}$. (When $\#\fC_w=1$, the width of $P_{\fw(w)}$ is strictly smaller than that of $P_w$.) If we let $\Phi(w) = (P, Q, \vv{\rho})=(P, Q, \offset_{P,Q}+\vv{\crho})$ where $\vv{\rho} = (\rho_1, \rho_2, \ldots)$, $\offset_{P,Q} = (s_1, s_2, \ldots)$, and $\vv{\crho}=(\crho_1, \crho_2, \ldots)$, then \cite[Theorem 8.1]{cpy18} shows that $h(C_1,C_2) = \rho_2-\rho_1-s_2 = \crho_2-\crho_1$ and $h(C_i,C_{i+1}) = h(D_{i-1}, D_i)$ for $i \in [2,m-1]$.

The distances between channels can be obtained from AMBC due to the following lemma.
\begin{lem} \label{lem:distch} Suppose that $\Phi(w) = (P, Q, \offset_{P,Q}+\vv{\crho})$ where $\vv{\crho}=(\crho_1, \crho_2, \ldots)$ and $\fC_w = (C_1, \ldots, C_m)$. Then for $i \in [1,m-1]$ we have $h(C_i, C_{i+1}) = \crho_{i+1} - \crho_{i}$.
\end{lem}
\begin{proof} When $i=1$ we already observed that $h(C_1, C_2) = \crho_{2}-\crho_1$ which follows from \cite[Theorem 8.1]{cpy18}. However, the second part of aforementioned theorem states that $h(C_i,C_{i+1}) = h(D_{i-1}, D_i)$, and thus the result follows from induction on $i$ and the fact that $\Phi(\fw(w)) = (P_{\geq2}, Q_{\geq2}, (\offset_{P,Q}+\vv{\crho})_{\geq 2})$.
%
\end{proof}
In particular, the indices $0=m_0, m_1, m_2, \ldots, m_l=m$ are chosen in a way that $\crho_i=\crho_j$ if and only if $i,j \in [1+m_{k-1}, m_{k}]$ for some $k \in [1,l]$, i.e. each river of $w$ corresponds to a subset of $[1,m]$ having the same values of $\vv{\crho}$.


For a partial permutation $w$ and a compatible stream $S$, the indexing river of $\bk_S(w)$ corresponding to $(w, S)$ is defined to be the collection of channels $C \in \bk_S(w)$ such that (if we fix a shift of $d^C_{\bk_S(w)}$ properly then) $d^C_{\bk_S(w)}(b)=m$ for any $b \in \bk_S(w)$ if and only if $m \in Z'_m$ where $Z'_m$ is as in \ref{sec:AMBC} (the set of elements labeled $m$). Then the collection is indeed a river. Moreover, when $(w,S) = (\fw(v), \st(v))$ for some partial permutation $v$ then the southwest channel of $v$ is contained in the indexing river of $(w,S)$ since the forward numbering (i.e. the southwest channel numbering) and the backward numberings are compatible in this case.

For a partial permutation $w$ and a compatible stream $S$ with the proper numbering $\tilde{d}: S \rightarrow \bZ$ and $d=d^{\bk,S}_w: w \rightarrow \bZ$, we say that $(x, w(x)) \in w$ is $N$-terminal (resp. $W$-terminal) with respect to $S$ if there exists $(y, S(y))\in S$ such that $\tilde{d}(y, S(y)) = d(x,w(x))+1$ and $x<y$ (resp. $w(x)<w(y)$). Also we say that $b \in w$ is $N$-terminating (resp. $W$-terminating) with respect to $S$ if there exists a path $b=b_0, \ldots, b_k \in w$ such that $b_k$ is $N$-terminal (resp. $W$-terminal) and $d(b_{i-1})+1=d(b_{i})$ for $i \in [1,k]$. 

By \cite[Lemma 14.22]{cpy18} and \cite[Lemma 16.15]{cpy18}, if $S$ is a stream compatible with $w$ and $R$ is a river of $w$ then either every $b\in \bigcup_{C\in R}C$ is $N$-terminating or every $b\in \bigcup_{C\in R}C$ is $W$-terminating, thus we may say that the river $R$ is either $N$-terminating or $W$-terminating. Furthermore, by \cite[Corollary 16.17]{cpy18} at most one river of $w$ is both $N$- and $W$-terminating. These notions are related to the indexing river of $\bk_S(w)$ as described in the lemma below. (Also see Figure \ref{fig:zigzag}.)

\begin{lem} \label{lem:indterm}Suppose that $v$ is a partial permutation, $S$ is a stream compatible with $v$, and $w = \bk_S(v)$. Then $(x,w(x)) \in w$ is in the indexing river corresponding to $(v,S)$ if and only if in the zigzag $Z$ containing $(x,w(x))$ (in the backward step of AMBC) the elements $(y, v(y)) \in Z\cap v$ such that $x<y$ (resp. $w(x)<v(y)$) are $W$-terminating (resp. $N$-terminating.)
\end{lem}
\begin{proof} The if part follows directly from \cite[Remark 16.7]{cpy18}. For the converse, suppose that $(x,w(x))$ is in the indexing river, say $R$. Let $(c, w(c)), (c',w(c')) \in w$ be such that $c$ is the minimum among the elements in $\bigcup_{C\in R}C \cap Z$ and $c'$ is the maximum. Then we have $(c',w(c')) \leq_{SW} (x, w(x))\leq_{SW} (c,w(c))$. By \cite[Lemma 16.5, 16.8]{cpy18} and their ``reflection statements along the anti-diagonal'' (or equivalently considering $w^{-1}$ instead), if we let $(y, v(y)) \in v$ (resp. $(y', v(y')) \in v$) be the most northeast $W$-terminating element (resp. the most southwest $N$-terminating element) of $Z \cap v$ then $c<y$ and $w(c)=v(y)$ (resp. $c'=y'$ and $w(c')<v(y')$). (If no such $(y, v(y))$ exists  then $(c, w(c))=(c', w(c'))=(x, w(x))$ is the most southwest element of $Z$, and if no such $(y', v(y'))$ exists then $(c, w(c))=(c', w(c'))=(x, w(x))$ is the most northeast element of $Z$.) Thus the result follows.
\end{proof}
As a result, elements of $v$ which are northeast of the indexing river of $\bk_S(v)$ corresponding to $(v,S)$ are $N$-terminating but not $W$-terminating, and a similar statement holds for balls southwest of the indexing river.

\begin{example} Figure \ref{fig:zigzag} describes a zigzag which appears in the backward step of AMBC. Here, the striped ball is an element of $S$, the white ones are the elements of $v$, and the shaded ones are the elements of $w=\bk_S(v)$. The curvy lines denote channels in the indexing river. The elements $C=(c, w(c)), C'=(c', w(c'))$ in Lemma \ref{lem:indterm} are depicted as in the diagram, and the white balls labeled $N$ (resp. $W$, resp. $NW$) are $N$-terminating but not $W$-terminating (resp. $W$-terminating but not $N$-terminating, resp. both $N$- and $W$-terminating).
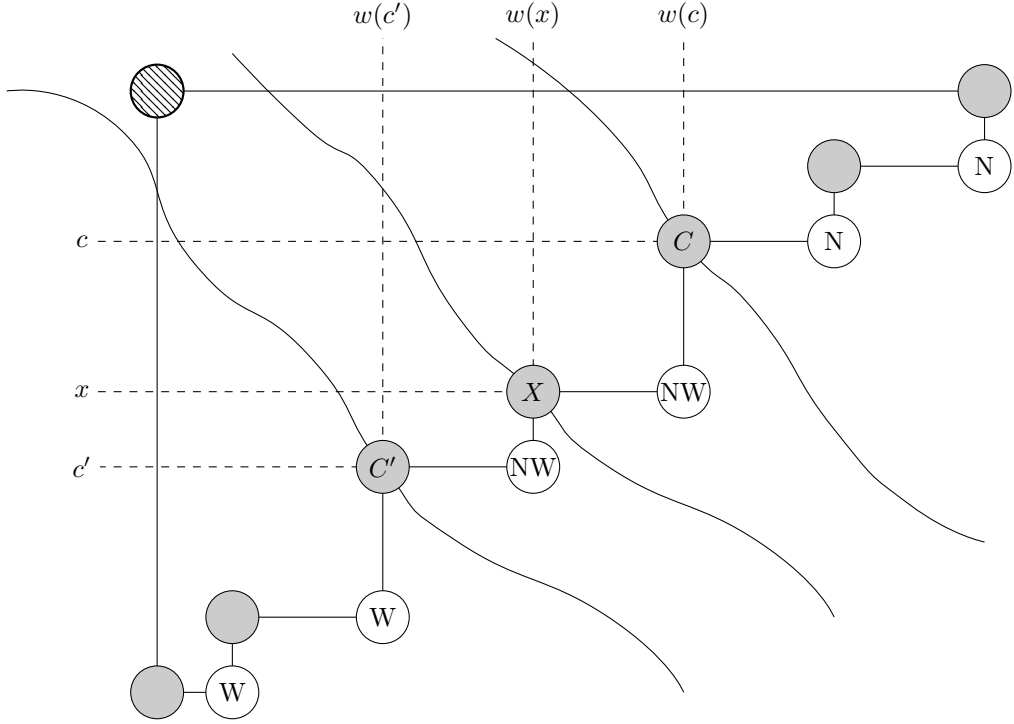
\begin{figure}[!htbp]
\begin{tikzpicture}

        \draw plot [smooth, tension=1] coordinates {(-2,0) (-.5,-.5) (0.5,-2.3) (2,-3.5) (3,-5) (4,-6) (6,-7) (7,-8)};
        \draw plot [smooth, tension=1] coordinates {(1,.5) (2,-.5) (3,-1.3) (4,-3) (5,-4) (6,-5) (8,-6) (9,-7)};
        \draw plot [smooth, tension=1] coordinates {(4.5,.7) (6,-.5) (7,-2) (8,-3) (9,-4.5) (10,-5.5) (11, -6)};
	\draw (0,0) --(0,-8) --(1,-8) --(1,-7) --(3,-7) --
		(3,-5) --(5,-5) --(5,-4) --(7,-4) --(7,-2)--
		(9,-2) --(9,-1) --(11,-1) --(11,0) -- cycle;
	\draw [dashed] (-1,-4) --(5,-4)--(5,1);
	\node [fill=white]at (-1,-4) {$x$};
	\node [fill=white]at (5,1) {$w(x)$};
	\draw [dashed] (-1,-2) --(7,-2)--(7,1);
	\node [fill=white]at (-1,-2) {$c$};
	\node [fill=white]at (7,1) {$w(c)$};
	\draw [dashed] (-1,-5) --(3,-5) -- (3, 1);
	\node [fill=white]at (-1,-5) {$c'$};
	\node [fill=white]at (3,1) {$w(c')$};
	\draw[fill=white, thick] (0,0) circle (10pt);
	\draw[pattern=north west lines] (0,0) circle (10pt);
	\draw[fill=white!80!black] (0,-8) circle (10pt);
	\draw[fill=white] (1,-8) circle (10pt) node {W};
	\draw[fill=white!80!black] (1,-7) circle (10pt);
	\draw[fill=white] (3,-7) circle (10pt) node {W};
	\draw[fill=white!80!black] (3,-5) circle (10pt) node {$C'$}; ;
	\draw[fill=white] (5,-5) circle (10pt) node {NW};
	\draw[fill=white!80!black] (5,-4) circle (10pt) node {$X$};
	\draw[fill=white] (7,-4) circle (10pt) node {NW};
	\draw[fill=white!80!black] (7,-2) circle (10pt) node {$C$};
	\draw[fill=white] (9,-2) circle (10pt) node {N};
	\draw[fill=white!80!black] (9,-1) circle (10pt);
	\draw[fill=white] (11,-1) circle (10pt) node {N};
	\draw[fill=white!80!black] (11,-0) circle (10pt);
\end{tikzpicture}
\caption{Indexing river and $N$/$W$-terminating elements}
\label{fig:zigzag}
\end{figure}
\end{example}

\subsection{Proper partial rotation and AMBC}
Here we define (proper) partial rotations and study their interactions with AMBC.
\begin{defn}
For a partial permutation $w$ and a stream $S \subset w$, we set $S=\{(x_i, w(x_i)) \mid i \in \bZ\}$ so that $i<j$ if and only if $x_i<x_j$. We define the partial rotation on $w$ with respect to $S$, denoted $\pr_S(w)$, as follows. 
\[
\pr_S(w)(a) =\left\{
\begin{aligned}
&w(x_{i+1}) & \textnormal{ if } a=x_i,
\\&w(a) & \textnormal{ otherwise}.
\end{aligned}
\right.
\]
Likewise, we also define the inverse partial rotation $\pr^{-1}_S(w)$ as follows.
\[
\pr_S^{-1}(w)(a) =\left\{
\begin{aligned}
&w(x_{i-1}) & \textnormal{ if } a=x_i,
\\&w(a) & \textnormal{ otherwise}.
\end{aligned}
\right.
\]
We often abbreviate the partial rotation (resp. inverse partial rotation) to PR (resp. IPR). We say that the partial rotation $\pr_S$ (resp. the inverse partial rotation $\pr_S^{-1}$) applied to $w$ is \emph{proper} if $S$ is the northeast (resp. southwest) channel of some river of $w$.
\end{defn}

In \cite{cpy18}, the term \emph{shift} is used in order to denote the PR/IPRs. Moreover, the notations $w\tbr{1}_S$, $w\tbr{-1}_S$, $S\tbr{1}$, and $S\tbr{-1}$ in \cite{cpy18} are translated to $\pr_S(w), \pr^{-1}_S(w), \pr_S(S),$ and $\pr^{-1}_S(S)$ in our language. When $R$ is a river of $w$ and $C$ (resp. $C'$) is the southwest (resp. northeast) channel of $R$, then $w\tbr{1}_R$ and $w\tbr{-1}_R$ in \cite{cpy18} are equal to $\pr_C(w)$ (resp. $\pr_{C'}^{-1}(w)$) which is by definition a proper partial rotation (resp. a proper inverse partial rotation). We stick to our notations in this paper, but allow to use $S\tbr{1}$ and $S\tbr{-1}$ instead of $\pr_S(S)$ and $\pr_S^{-1}(S)$.

Let us describe the relation between proper PR/IPRs and AMBC.
\begin{thm} \label{thm:rho} For a partial permutation $w$ with $\fC_w = (C_1, \ldots, C_r)$, assume that $R$ is a river of $w$ whose southwest (resp. northeast) channel is $C_{r'}$ (resp. $C_{r}$). (Note that both $\pr_{C_r}$ and $\pr^{-1}_{C_{r'}}$ are proper with respect to $w$.  Suppose that $\Phi(w) = (P, Q, \vv{\rho})$. Then we have
\[\Phi(\pr_{C_r}(w))=(P, Q, \vv{\rho}+\vv{e_r}) \quad \textnormal{ and } \quad \Phi(\pr^{-1}_{C_{r'}}(w))=(P, Q, \vv{\rho}-\vv{e_{r'}}).\]
Here, $\vv{e_i}$ is a vector whose $i$-th coordinate is 1 and 0 elsewhere.
%
%
\end{thm}
\begin{proof}
We first prove $\Phi(\pr_{C_r}(w))=(P, Q, \vv{\rho}+\vv{e_r})$. 
Assume that $r=1$ (which forces that $r'=1$ as well). Let us write $(v,S)=(\fw(w),\st(w))$ so that $w=\bk_S(v)$, $\Phi(v) = (P_{\geq 2}, Q_{\geq 2}, \vv{\rho}_{\geq2})$, and $a(S)=\rho_1$.
Since $C_1$ is the southwest channel of $w$, $R$ is the indexing river of $w$ corresponding to $(v,S)$. Thus by \cite[Theorem 16.9]{cpy18} we have $\pr_{C_1}(w) =  \bk_{S\tbr{1}}(v)$. Since $h(C_1, C_2)>0$  by assumption, we have $\crho_1<\crho_2$ by Lemma \ref{lem:distch}, which implies that $\vv{\rho}+\vv{e_1}$ is still dominant with respect to the pair $(P, Q)$. Thus from the definition of backward algorithm we see that $\Phi(\pr_{C_1}(w)) = \Phi(\bk_{S\tbr{1}}(v)) = (P, Q, \vv{\rho}+\vv{e_1})$ (note that $a(S\tbr{1})=\rho_1+1$) as desired.

In general, we successively define $(v_1,S_1) = (\fw(w),\st(w) )$ and $(v_{i},S_{i}) = (\fw(v_{i-1}),\st(v_{i-1}))$ for $i \in [2,r]$. Let $C_r=D^0\subset w, D^1\subset v_1, \ldots, D^{r}\subset v_{r}$ be the $(m+1-r)$-th northeast channel in $\fC_w, \fC_{v_1}, \ldots, \fC_{v_{r}}$, respectively. Note that each of them is the northeast channel of some river, say $D^i \in R^i$, by Lemma \ref{lem:distch} since so is $C_r=D^0\in R^0=R$. 


The above argument for $r=1$ case shows that $\pr_{D^{r-1}}(v_{r-1}) = \bk_{S_r\tbr{1}}(v_r)$ and thus $\Phi(\pr_{D^{r-1}}(v_{r-1})) = (P_{\geq r}, Q_{\geq r}, (\vv{\rho}+\vv{e_r})_{\geq r})$. Now recall that the southwest channel of $v_{r-2}$ is in the indexing river corresponding to $(v_{r-1},S_{r-1})$. In particular, any ball in $v_{r-1}$ is $N$-terminating by Lemma \ref{lem:indterm}. Thus (the reflection statement along the anti-diagonal of) \cite[Proposition 16.39]{cpy18} shows that $\bk_{S_{r-1}}(\pr_{D^{r-1}}(v_{r-1})) =\pr_{D^{r-2}}(v_{r-2})$. (Here, $D^{r-2}$ is the northeast channel of $R^{r-2}$, and $R^{r-2}$ contains the most southwest channel among the ones of $v_{r-2}$ that are northeast of $D^{r-1}$.) Thus it follows that $\Phi(\pr_{D^{r-2}}(v_{r-2})) = \Phi(\bk_{S_{r-1}}(\pr_{D^{r-1}}(v_{r-1})) ) = (P_{\geq r-1}, Q_{\geq r-1}, (\vv{\rho}+\vv{e_r})_{\geq r-1})$.

Now we iterate this process and eventually reach the conclusion that $\pr_{C_r}(w)= \bk_{S_1}(\pr_{D^1}(v_1)) = (\bk_{S_1}\circ \bk_{S_2})(\pr_{D^2}(v_2)) = \cdots = (\bk_{S_1} \circ \cdots \circ \bk_{S_{r-1}}\circ\bk_{S_r\tbr{1}})(v_r)$. Therefore from the definition of backward algorithm we see that $\Phi(\pr_{C^r}(w))=\Phi((\bk_{S_1} \circ \cdots \circ \bk_{S_{r-1}}\circ\bk_{S_r\tbr{1}})(v_r))= (P, Q, \vv{\rho}+\vv{e_r})$ as desired. (Here we use the fact that $\vv{\rho}+\vv{e_r}$ is dominant with respect to $(P,Q)$.)

To prove the second claim $\Phi(\pr^{-1}_{C_{r'}}(w))=(P, Q, \vv{\rho}-\vv{e_{r'}})$ one may similarly argue as above. The case when $r'=1$ is almost the same as the above, and for the inductive argument one needs to use (the reflection statement along the anti-diagonal of) \cite[Proposition 16.38]{cpy18} instead of  \cite[Proposition 16.39]{cpy18}. We omit the details.
%
%
%
%
%
\end{proof}

\begin{example} \label{ex:ppr} Suppose that $w=[6,1,18,3,19,24,12,15,17,10] \in \extS{10}$ as in Example \ref{ex:fCw}. Then we have $\Phi(w) = (P, Q, \vv{\rho})=(((1,3,10),(2,5,6),(4,7,9),(8)), ((3,5,6),(7,8,9),(1,4,10), (2)), (2,3,2,0))$ and $\offset_{P,Q}=(0,1,1,0)-(0,0,2,0) = (0,1,-1,0)$. Thus $\vv{\crho}=\vv{\rho}-\offset_{P,Q}=(2,2,3,0)$. (This again shows that there are two rivers $R_1 \owns C_1, C_2$ and $R_2 \owns C_3$.)
If we calculate $\pr_C(w)$ and $\pr_C^{-1}(w)$ for each $C \in \fC_w$ when the operations are proper then
\begin{align*}
\pr_{C_1}^{-1}(w) &= [6,0,18,1,19,24,12,15,17,3]
\\&\xmapsto{\Phi} (((1,3,10),(2,5,6),(4,7,9),(8)), ((3,5,6),(7,8,9),(1,4,10), (2)), (1,3,2,0)),
\\\pr_{C_2}(w) &= [6,1,18,3,19,24,15,17,22,10]
\\&\xmapsto{\Phi} (((1,3,10),(2,5,6),(4,7,9),(8)), ((3,5,6),(7,8,9),(1,4,10), (2)), (2,4,2,0)),
\\\pr_{C_3}(w) &= [6,1,14,3,18,19,12,15,17,10]
\\&\xmapsto{\Phi} (((1,3,6,0), (2,5,9), (4,7),(8)), ((3,5,6,9), (4,7,8), (1,10), (2)),(2,3,1,0)),
\\\pr_{C_3}^{-1}(w) &= [6,1,19,3,24,28,12,15,17,10]
\\&\xmapsto{\Phi} (((1,3,6,0), (2,5,9), (4,7),(8)), ((3,5,6,9), (4,7,8), (1,10), (2)),(2,3,3,0)),
\end{align*}
as expected.
\end{example}

\subsection{Proper partial rotation and sign insertion}
Here we show that the proper PR/IPRs do not affect the result of sign insertion as the following theorem states.
\begin{thm} \label{thm:ppr}
Suppose that $\pr_C$ and $\pr_{C'}^{-1}$ are proper with respect to $w\in \extSn$ for some $C, C' \subset w$. Then we have $\sgn_\fQ(w) = \sgn_{\fQ}(\pr_C(w)) = \sgn_{\fQ}(\pr_{C'}^{-1}(w))$.
\end{thm}
The main tool to prove this theorem is the following Diamond Lemma discussing interactions between right star operations and proper PR/IPRs.

\begin{lem} [Diamond lemma] \label{lem:diamond}
Suppose that we are given a partial permutation $w$ with $\fC_w=(C_1, \ldots, C_m)$, $p \in [1,n]$, and $q\in [1,m]$. Assume that:
\begin{enumerate}[label=\textbullet]
\item Either $w(p-1)$ is between $w(p)$ and $w(p+1)$ so that a right star operation $w \mapsto w^*$ that switches $w(p)$ and $w(p+1)$ is well-defined, or $w(p+1)$ is between $w(p-1)$ and $w(p)$ so that a right star operation $w \mapsto w^*$ that switches $w(p-1)$ and $w(p)$ is well-defined. Let us call it ``the right star operation centered at $p$''.
\item $C_q$ is a northeast channel of a river of $w$.
\end{enumerate}
Then, we have:
\begin{enumerate}[label=\textbullet]
\item The $q$-th southwest channel $C_q^*$ of $\fC_{w^*}$ is a northeast channel of a river of $w^*$.
\item The right star operation centered at $p$ is well-defined for $\pr_{C_q}(w)$, say $\pr_{C_q}(w)^*$.
\end{enumerate}
Furthermore, we have $\pr_{C_q}(w)^* = \pr_{C_q^*}(w^*)$. In other words, we may ``canonically complete the commutative diagram'' when we are given $w\mapsto w^*$ and $w \mapsto \pr_{C_q}(w)$. The same result also holds when one replaces proper partial rotations with proper inverse partial rotations and/or replaces right star operations with multiplication by $\shift$ on the right.
\end{lem}
Since star operations and multiplication by $\shift$ preserve each two-sided cell, or equivalently the Greene-Kleitman partition of a Shi poset, it makes sense to take the $q$-th southwest channel of $w^*$ and $w\cdot\shift$ on the lemma above. Also note that the statement for multiplication by $\shift$ on the right is straightforward because the channels of $w\cdot \shift$ are simply the shifts of those of $w$.

\begin{example}
Let $w=[6,1,18,3,19,24,12,15,17,10]$ as in Example \ref{ex:fCw} and \ref{ex:ppr}. Then the star operation $w \mapsto w^*$ centered at 7 is well defined and $w^* = [6,1,18,3,19,12,24,15,17,10].$ Let us consider $\pr_{C_2}(w) = [6,1,18,3,19,24,15,17,22,10].$ Then $\pr_{C_2}(w) \mapsto \pr_{C_2}(w)^*$ centered at 7 is well-defined and $\pr_{C_2}(w)^* = [6,1,18,3,19,15,24,17,22,10].$ Moreover, $w^* \mapsto\pr_{C_2}(w)^*$ is the proper partial rotation with respect to the channel $[\emptyset, \emptyset, \emptyset, \emptyset, \emptyset, 12, \emptyset, 15, 17, \emptyset]$.
\end{example}

First let us discuss how to use Lemma \ref{lem:diamond} to prove Theorem \ref{thm:ppr}.
\begin{proof}[Proof of Theorem \ref{thm:ppr}, assuming Lemma \ref{lem:diamond}]
It suffices only to prove for the proper partial rotation, i.e. $\sgn_\fQ(w) = \sgn_{\fQ}(\pr_C(w))$. Let us set $w_0=w, \fP_0=\fQ_0=\emptyset$ and sequentially define $(i, \fP_i, \fQ_i, w_i) = \digamma(i-1, \fP_{i-1}, \fQ_{i-1}, w_{i-1})$ for $i \in [1,N]$ where $w_N=\emptyset$. Thus in particular we have $\sgn_\fP(w) = \fP_N$ and $\sgn_\fQ(w) = \fQ_N$. Let us set $v_i = \fP_i \con w_i$ for $i\in [0,N]$. Note that either $v_i= v_{i-1}$ or $v_{i}$ is obtained from $v_{i-1}$ by applying right star operations followed by multiplication by $\shift$ on the right (cf. the proof of Lemma \ref{lem:signleft}).

Let $\tilde{w}=\pr_C(w)$ where $\pr_C$ applied to $w$ is proper. Then by Lemma \ref{lem:diamond}, we may (uniquely) find $ \tilde{w}=\tilde{v}_0, \tilde{v}_1, \ldots, \tilde{v}_N$ such that $v_i \mapsto \tilde{v}_i$ is a proper partial rotation, $v_i = v_{i-1}$ if and only if $\tilde{v}_i = \tilde{v}_{i-1}$, and if $\tilde{v}_i \neq \tilde{v}_{i-1}$ then $\tilde{v}_i$ is obtained from $\tilde{v}_{i-1}$ by a series of right star operations followed by multiplication by $\shift$ on the right. We claim that there exist $\tilde{\fP}_i, \tilde{w}_i$ for $i \in [0,N]$ such that $\tilde{v}_i = \tilde{\fP}_i \con \tilde{w}_i$ for $i \in [0,N]$ and $(i, \tilde{\fP}_{i}, \fQ_{i}, \tilde{w}_{i})=\digamma(i-1, \tilde{\fP}_{i-1}, \fQ_{i-1}, \tilde{w}_{i-1})$ for $i \in [1,N]$. (In particular the lengths of $\tilde{\fP}_i, \fP_i,$ and $\fQ_i$ are all equal.)

 We prove by induction on $i$. It is obvious when $i=0$ by setting $\fP_0 = \emptyset$ and $\tilde{w}_0 = \tilde{w}$, so suppose that it is true up to $i-1$. Let $p$ be the length of $\fP_{i-1}$. Then we have $v_{i-1}(1)<\cdots<v_{i-1}(p)$ and $\tilde{v}_{i-1}(1)<\cdots<\tilde{v}_{i-1}(p)$. If $v_{i-1}(p)<v_{i-1}(p+1)$, then it means that $v_{i-1}=v_i$, i.e. $v_{i-1} \mapsto v_i$ is a trivial step, which implies $\tilde{v}_{i-1} = \tilde{v}_{i}$ by assumption. We claim that $\tilde{v}_{i-1}(p)<\tilde{v}_{i-1}(p+1)$ so that we may set $\tilde{\fP}_i$ and $\tilde{w}_i$ similar to $\fP_i$ and $w_i$. Indeed, if $\tilde{v}_{i-1}(p)>\tilde{v}_{i-1}(p+1)$ then it is possible to apply the right star operation centered at $p$ on $\tilde{v}_{i-1}$. However, by Lemma \ref{lem:diamond}, it means that the same star operation can be applied to $v_{i-1}$, which is a contradiction.

It remains to consider the case when $v_{i-1}(p)>v_{i-1}(p+1)$. Then $v_{i}$ is obtained from $v_{i-1}$ by applying the right star operations centered at $p, p-1, \ldots, 2$, respectively, followed by multiplication by $\shift$ on the right. By Lemma \ref{lem:diamond}, the same process applies to $\tilde{v}_{i-1}$ to get $\tilde{v}_i$. However, as $\tilde{v}_{i-1}(1)<\cdots<\tilde{v}_{i-1}(p)$ and $\tilde{v}_{i-1}(p)>\tilde{v}_{i-1}(p+1)$ by assumption, one can easily show that this is the usual ``bumping process'' (which inserts $\tilde{v}_{i-1}(p+1)$ to $\tilde{\fP}_{i-1}$) followed by multiplication by $\shift$ on the right. This is clearly a valid step of sign insertion algorithm. Thus, if we set $\tilde{\fP}_{i}$ to be the first $p$-th characters of $\tilde{v}_i$ and $\tilde{w}_i$ to be the remainder then the induction step is also valid in this case. It suffices for the proof.
\end{proof}

\begin{example} Here we compare the sign insertion of $w=[6,1,18,3,19,24,12,15,17,10]$ and $\pr_{C_2}(w)=[6,1,18,3,19,24,15,17,22,10]$. The underlined numbers are where the proper partial rotation is applied. It is clear that the sign insertion processes for $w$ and $\pr_{C_2}(w)$ are parallel.
\begin{longtable}{|c|c|c|c|}
\hline
$i$&$\fP$ & $\fQ$ & $w$\\
\hline
0&$\emptyset$ & $\emptyset$ & 6 1 18 3 19 24 \underline{12} \underline{15} \underline{17} 10\\
1&6&1&1  18 3 19 24 \underline{12} \underline{15} \underline{17} 10\\
2&1&1&18 3 19 24 \underline{12} \underline{15} \underline{17} 10 16\\
3&1  18&1 3&3 19 24 \underline{12} \underline{15} \underline{17} 10 16\\
4&1  3&1 3&19 24 \underline{12} \underline{15} \underline{17} 10 16 28\\
5&1  3  19&1 3 5&24 \underline{12} \underline{15} \underline{17} 10 16 28\\
6&1  3  19 24&1 3 5  6&\underline{12} \underline{15} \underline{17} 10 16 28\\
7&1  3  \underline{12} 24&1 3 5  6&\underline{15} \underline{17} 10 16 28 29\\
8&1  3  \underline{12} \underline{15}&1 3 5  6&\underline{17} 10 16 28 29 34\\
9&1  3  \underline{12} \underline{15} \underline{17}&1 3 5  6 7&10 16 28 29 34\\
10&1  3  10 \underline{15} \underline{17}&1 3 5  6 7&16 28 29 34 \underline{22}\\
11&1  3  10 \underline{15} 16&1 3 5  6 7&28 29 34 \underline{22} \underline{27}\\
12&1  3  10 \underline{15} 16 28&1 3 5  6 7 12&29 34 \underline{22} \underline{27}\\
13&1  3  10 \underline{15} 16 28 29&1 3 5  6 7 12 13&34 \underline{22} \underline{27}\\
14&1  3  10 \underline{15} 16 28 29 34&1 3 5  6 7 12 13 14&\underline{22} \underline{27}\\
15&1  3  10 \underline{15} 16 \underline{22} 29 34&1 3 5  6 7 12 13 14&\underline{27} 38\\
16&1  3  10 \underline{15} 16 \underline{22} \underline{27} 34&1 3 5  6 7 12 13 14&38 39\\
17&1  3  10 \underline{15} 16 \underline{22} \underline{27} 34 38&1 3 5  6 7 12 13 14 17&39\\
18&1  3  10 \underline{15} 16 \underline{22} \underline{27} 34 38 39&1 3 5  6 7 12 13 14 18&$\emptyset$\\
\hline
\end{longtable}
\begin{longtable}{|c|c|c|c|}
\hline
$i$&$\fP$ & $\fQ$ & $w$\\
\hline
0&$\emptyset$ & $\emptyset$ & 6 1 18 3 19 24 \underline{15} \underline{17} \underline{22} 10\\
1&6&1&1 18 3 19 24 \underline{15} \underline{17} \underline{22} 10\\
2&1&1&18 3 19 24 \underline{15} \underline{17} \underline{22} 10 16\\
3&1  18&1 3&3 19 24 \underline{15} \underline{17} \underline{22} 10 16\\
4&1  3&1 3&19 24 \underline{15} \underline{17} \underline{22} 10 16 28\\
5&1  3  19&1 3 5&24 \underline{15} \underline{17} \underline{22} 10 16 28\\
6&1  3  19 24&1 3 5  6&\underline{15} \underline{17} \underline{22} 10 16 28\\
7&1  3  \underline{15} 24&1 3 5  6&\underline{17} \underline{22} 10 16 28 29\\
8&1  3  \underline{15} \underline{17}&1 3 5  6&\underline{22} 10 16 28 29 34\\
9&1  3  \underline{15} \underline{17} \underline{22}&1 3 5  6 7&10 16 28 29 34\\
10&1  3  10 \underline{17} \underline{22}&1 3 5  6 7&16 28 29 34 \underline{25}\\
11&1  3  10 16 \underline{22}&1 3 5  6 7&28 29 34 \underline{25} \underline{27}\\
12&1  3  10 16 \underline{22} 28&1 3 5  6 7 12&29 34 \underline{25} \underline{27}\\
13&1  3  10 16 \underline{22} 28 29&1 3 5  6 7 12 13&34 \underline{25} \underline{27}\\
14&1  3  10 16 \underline{22} 28 29 34&1 3 5  6 7 12 13 14&\underline{25} \underline{27}\\
15&1  3  10 16 \underline{22} \underline{25} 29 34&1 3 5  6 7 12 13 14&\underline{27} 38\\
16&1  3  10 16 \underline{22} \underline{25} \underline{27} 34&1 3 5  6 7 12 13 14&38 39\\
17&1  3  10 16 \underline{22} \underline{25} \underline{27} 34 38&1 3 5  6 7 12 13 14 17&39\\
18&1  3  10 16 \underline{22} \underline{25} \underline{27} 34 38 39&1 3 5  6 7 12 13 14 18&$\emptyset$\\
\hline
\end{longtable}
\end{example}

\subsection{Proof of Diamond Lemma}
It remains to prove Lemma \ref{lem:diamond}. Recall the definition of $\fC_w$ in \ref{sec:chan}.
\begin{lem} \label{lem:knuth} For a partial permutation $w$, right star operations and multiplication by $\shift$ on the right do not change the distances between channels in $\fC_w$.
\end{lem}
\begin{proof}
This follows from Lemma \ref{lem:distch} together with the last part of Section \ref{sec:cells}. (Here, it is crucial that $\vv{\delta}(T,s)=(\delta_1, \ldots, \delta_l)$ satisfies that $\delta_i=\delta_j$ whenever $\lambda_i=\lambda_j$.)
\end{proof}
The following lemma describes the reason why the proper PR/IPRs are useful.
\begin{lem} \label{lem:NESW}
If $C \subset w\in \extSn$ is a channel then $\pr_C(w)$ (resp. $\pr_C^{-1}(w)$) and $w$ are in the same two-sided cell if and only if $S$ is the northeast (resp. southwest) channel of a river of $w$. (In this case such element is in the same left and right cell of $w$ by Theorem \ref{thm:rho}.)
\end{lem}
\begin{proof}
Here we prove only for partial rotations; the inverse case can be proved similarly. The if part follows from Theorem \ref{thm:rho} and it remains to prove the only if part.
Assume that $C\subset w$ is a channel which is not the northeast in its river, but its partial rotation yields a permutation in the same two-sided cell. By assumption there exists a channel $C'\neq C$ such that $C\leq_{SW} C'$ and $h(C, C')=0$, i.e. $d^C$ and $d^{C'}$ are equal up to shift by \cite[Definition 3.16]{cpy18}. We fix shifts of $d^C$ and $d^{C'}$ such that they coincide, and let $c_i \in C$ and $b_i \in C'$ be balls such that $d^C(c_i) = d^{C'}(c_i) = d^C(b_i) = d^{C'}(b_i)=i$ for $i \in \bZ$. Note that by assumption either $b_i =c_i$ or $c_i \leq_{SW} b_i$, and there exists $i\in \bZ$ such that $b_i\neq c_i$.



We claim that there exists $j \in \bZ$ such that $b_j=(y, w(y)) \in C'-C$ (thus $c_j \leq_{SW} b_j$) and $c_{j-1}=(x, w(x))\in C$ satisfies $x<y$ and $w(x)<w(y)$. If $c_{j-1}=b_{j-1}$ then the claim is obvious, thus it suffices to consider the case when $C\cap C' = \emptyset$. However, if the assumption is false for any $j \in \bZ$ then we may lower the labels of the balls in $C'$ by 1 and the numbering still satisfies the monotone property, which is a contradiction. Thus such $j$ always exists.

Now for $i \in \bZ$ let $d_i$ be a ball in the partial rotation of $C$ so that $c_i$ and $d_i$ share the same $x$-coordinate, or equivalently the $y$-coordinate of $d_{i}$ is the same as that of $c_{i+1}$. Let $b_j \in C' -C$ be as above. Then it follows that $x$-coordinate of $b_j$ is between those of $c_{j-1}$ and $c_j$ and the $y$-coordinate of $b_j$ is bigger than that of $c_j$, which in turn means that $d_{j-1}$ is northeast of $b_j$. In addition, similar argument shows that there exists $k \in \bZ$ such that $b_{k}$ is northeast of $d_{k}$.

By adjusting the shift of $d^C$ and $d^{C'}$ necessary, we may assume that $1 \leq j \leq k\leq r$ where $r$ is the density of $C$. Then the conditions above imply that $(\{b_i \mid j\leq i\leq k\}\cup\{d_i \mid 1\leq i\leq j-1 \textnormal{ or } k\leq i\leq r\})+\bZ(n,n) \subset w$ is a stream of density $r+1$, which contradicts the assumption that $\pr_C(w)$ is in the same two-sided cell of $w$ (cf. \cite[Corollary 11.5]{cpy18}). Thus the result follows.
\end{proof}
\begin{example} Suppose that $w=[6,1,18,3,19,24,12,15,17,10]$ as in Example \ref{ex:ppr} and \ref{ex:fCw}. We already calculated $\pr_{C_1}^{-1}(w)$, $\pr_{C_2}(w)$, $\pr_{C_3}(w)$, and $\pr_{C_3}^{-1}(w)$, and their images under $\Phi$. Let us calculate the remaining possibilities.
\begin{align*}
\pr_{C_1}(w) &= [6,3,18,10,19,24,12,15,17,11]
\\&\xmapsto{\Phi} (((1,3,6,10),(2,5,9),(4,7),(8)), ((3,5,6,9),(4,7,8),(1,10), (2)), (3,4,1,0)),
\\\pr_{C_2}^{-1}(w) &=[6,1,18,3,19,24,7,12,15,10]
\\&\xmapsto{\Phi} (((1,3,7,10),(2,5,6),(4,8,9)), ((1,3,5,6),(7,8,9),(2,4,10)), (1,3,2)),
\\\pr_{C_2'}(w) &=[12,1,18,3,19,24,15,16,17,10]
\\&\xmapsto{\Phi} (((1,3,7,10),(2,5,6),(4,8,9)), ((1,3,5,6),(7,8,9),(2,4,10)), (3,3,2)),
\\\pr_{C_2'}^{-1}(w) &= [5,1,18,3,19,24,6,12,17,10]
\\&\xmapsto{\Phi} (((1,3,6,10),(2,5,9),(4,7),(8)), ((3,5,6,9),(4,7,8),(1,10), (2)), (1,4,1,0)).
\end{align*}
Note that $w\in \tsc_{(3,3,3,1)}$ but these elements are not contained in this two-sided cell.
\end{example}

We are ready to prove Lemma \ref{lem:diamond}. To recap, we assume that to an affine permutation $w$ one can apply two operations: (1) a right star operation centered at $p$ for $p \in [1,n]$; (2) a proper PR/IPR applied to the $q$-th southwest channel $C$ in $\fC_w$.
Let us denote by $w^*$ the outcome of the right star operation and by $\tilde{w}$ the outcome of such a proper PR/IPR. First we claim the following. 
\begin{lem}
Suppose that there exists $\tilde{w}^* \in \extSn$ such that $\tilde{w} \mapsto \tilde{w}^*$ is again the right star operation centered at $p$ and $w^* \mapsto \tilde{w}^*$ is a PR/IPR. Then the latter operation is proper and applied to the $q$-th southwest channel of $w^*$.
\end{lem}
\begin{proof}
By Theorem \ref{thm:rho},  $\Phi(w)$ and $\Phi(\tilde w)$ differ by a single value of $\vv{\rho}$ on the top block. Since $w, \tilde{w}, w^*, \tilde{w}^*$ all lie within the same two-sided cell by assumption, Lemma \ref{lem:NESW} tells us that the PR/IPR $w^* \mapsto \tilde {w}^*$ should be proper. Since the distances between channels of $w^*$ (resp. $\tilde{w}^*$) are the same as those of $w$ (resp. $\tilde{w}$) by Lemma \ref{lem:knuth}, we see that the PR/IPR $w^* \mapsto \tilde {w}^*$ should be applied to the $q$-th channel of $w^*$.
\end{proof}
Thus, it is enough to show that the commutative diagram can be completed by a PR/IPR (without checking properness) applied to a channel of $w^*$ (without checking that it is the $q$-th southwest channel). Meanwhile, without loss of generality we may assume that $p \in [2,n-1]$; if $p\in \{1,n\}$, then we may multiply $\shift$ on the right so that the right star operation is centered in $[2,n-1]$; this process also respects proper PR/IPRs. (See the comment after Lemma \ref{lem:diamond}. Also, here $n \geq 3$ since the right star operation is not well-defined for $n \leq 2$.) Moreover, it suffices only to consider partial rotations because the inverse case can be deduced by considering $w^{-1}$.

We write $\tilde{w}=\pr_C(w)$ where $C \subset w$ is a northeast channel of a river. We set $a, b, c\in \bZ$ to be the $y$-coordinates of some elements in $C$ such that $w^{-1}(a)<w^{-1}(b)<w^{-1}(c)$, and we assume that no other elements in $C$ are between them, i.e. if $(y,w(y)) \in C$ and $a<w(y)<c$ then $b=w(y)$. (If the density of $C$ is 2 (resp. 1) then we have $c=a+n$ (resp. $c=b+n=a+2n$).)

We argue case-by-case based on the size of $[p-1,p+1]\cap X$. Here we will verify that (1) the right star operation centered at $p$ is well-defined for $\tilde{w}$, and (2) there exists a channel $C' \subset w^*$ such that its corresponding partial rotation coincide with $\tilde{w}^*$.

\subsubsection*{\textbf{Case 0.}} Suppose that $\#([p-1, p+1] \cap X)=0$. In this case the two operations simply commute, and the commutative diagram can be trivially completed.
 
\subsubsection*{\textbf{Case 1.}} Suppose that $\#([p-1, p+1] \cap X)=1$.
\begin{enumerate}
\item The right star operation looks like (here either side can be $w$):
\[
 \cdots x b y \cdots c \cdots \leftrightsquigarrow  \cdots b x y \cdots c \cdots.
\]
\begin{enumerate}
\item Suppose that $x < y < b$. Since $x < y < c$, after the partial rotation we get a valid right star operation
\[
\cdots b \cdots x c y \cdots \leftrightsquigarrow \cdots b \cdots  c x y \cdots ,
\]
and a valid channel $C' \subset w^*$ whose $y$-coordinates are identical to $C$ so that the commutative diagram can be completed.
\item Suppose that $b < y < x$. If $c>y$ then $y$ and its translates can be inserted to $C$ to yield a longer stream in $w$, which is a contradiction. Thus $c<y<x$ and after the partial rotation we get a valid right star operation
\[
\cdots b \cdots x c y \cdots \leftrightsquigarrow \cdots b \cdots c x y \cdots,
\]
and a valid channel  $C' \subset w^*$ whose $y$-coordinates are identical to $C$ so that the commutative diagram can be completed.
\end{enumerate}
\item The right star operation looks like (here either side can be $w$):
\[
 \cdots y x b \cdots c \cdots \leftrightsquigarrow  \cdots y b x \cdots c \cdots.
\] 
\begin{enumerate}
\item Suppose that $x < y < b$. Since $x < y < c$, after the partial rotation we get a valid right star operation
\[
\cdots b \cdots yxc \cdots \leftrightsquigarrow \cdots b \cdots  ycx \cdots ,
\]
and a valid channel  $C' \subset w^*$ whose $y$-coordinates are identical to $C$ so that the commutative diagram can be completed.
\item Suppose that $b < y < x$. If $c>y$ then $y$ and its translates can be inserted to $C\tbr{1}$ to yield a longer stream in $\tilde{w}$, which is a contradiction. Thus $c<y<x$ and after the partial rotation we get a valid right star operation
\[
\cdots b \cdots yxc \cdots \leftrightsquigarrow \cdots b \cdots  ycx \cdots ,
\]
and a valid channel  $C' \subset w^*$ whose $y$-coordinates are identical to $C$ so that the commutative diagram can be completed.
\end{enumerate}
\item The right star operation looks like (here either side can be $w$):
\[
 \cdots y x b \cdots c \cdots \leftrightsquigarrow  \cdots x y b \cdots c \cdots.
\]
Without loss of generality we may assume that $x<y$, i.e. $x<b<y$. If $c> y$  then $y$ and its translates can be inserted to $C\tbr{1}$ to yield a longer stream in $\tilde{w}$, which is a contradiction. Thus $x<c<y$ and after the partial rotation we get a valid right star operation
\[
\cdots b \cdots y x c  \cdots \leftrightsquigarrow \cdots b \cdots x y c \cdots,
\]
and a valid channel  $C' \subset w^*$ whose $y$-coordinates are identical to $C$ so that the commutative diagram can be completed.
\item The right star operation looks like (here either side can be $w$):
\[
 \cdots by x \cdots c \cdots \leftrightsquigarrow  \cdots bx y \cdots c \cdots.
\]
Without loss of generality we may assume that $x<y$, i.e. $x<b<y$. If $c> y$  then $y$ and its translates can be inserted to $C$ to yield a longer stream in $w$, which is a contradiction. Thus $x<c<y$ and after the partial rotation we get a valid right star operation
\[
\cdots b \cdots cy x  \cdots \leftrightsquigarrow \cdots b \cdots cx y \cdots,
\]
and a valid channel  $C' \subset w^*$ whose $y$-coordinates are identical to $C$ so that the commutative diagram can be completed.
\end{enumerate}

\subsubsection*{\textbf{Case 2.}} 
Suppose that $\#([p-1, p+1] \cap X)=2$.
\begin{enumerate}
\item The right star operation looks like:
\[
w = \cdots a z b \cdots c \cdots \rightsquigarrow w^*=\cdots z a b \cdots c \cdots.
\]
Here we should have $a<b<z$. If $z < c$, then $z$ and its translates can be inserted to $C\tbr{1}$ to yield a longer stream in $\tilde{w}$, which is a contradiction. Thus $b<c<z$ and after the partial rotation we get a valid right star operation
\[
\tilde{w} = \cdots a\cdots b z c \cdots  \rightsquigarrow \tilde{w}^*=\cdots a\cdots z b c \cdots
\]
and a valid channel  $C' \subset w^*$ whose $y$-coordinates are identical to $C$ so that the commutative diagram can be completed.
\item The right star operation looks like:
\[
w = \cdots zab \cdots c \cdots \rightsquigarrow w^*=\cdots azb \cdots c \cdots.
\]
Here we should have $a<b<z$. 
\begin{enumerate}
\item If $c < z$, then after the partial rotation we get a valid right star operation
\[
\tilde{w} = \cdots a\cdots zbc \cdots  \rightsquigarrow \tilde{w}^*=\cdots a\cdots bzc \cdots
\]
and a valid channel  $C' \subset w^*$ whose $y$-coordinates are identical to $C$ so that the commutative diagram can be completed.
\item If $c > z$, then after the partial rotation we get a valid right star operation
\[
\tilde{w} = \cdots a\cdots zbc \cdots  \rightsquigarrow \tilde{w}^*=\cdots a\cdots zcb \cdots
\]
and a valid channel $C' = \ldots, a, z, c, \ldots\subset w^*$ (i.e. replacing $b$ with $z$ in $C$) so that the commutative diagram can be completed.
\end{enumerate}
\item The right star operation looks like (here either side can be $w$):
\[
\cdots axb \cdots c \cdots \leftrightsquigarrow  \cdots abx \cdots c \cdots.
\] 
Here we should have $x<a<b<c$. After the partial rotation we get a valid right star operation
\[
\cdots a \cdots bxc \cdots  \leftrightsquigarrow  \cdots a\cdots bcx \cdots
\] 
and a valid channel  $C' \subset w^*$ whose $y$-coordinates are identical to $C$ so that the commutative diagram can be completed.
\item The right star operation looks like:
\[
w= \cdots aby \cdots c \cdots \rightsquigarrow  w^*=\cdots bay \cdots c \cdots.
\] 
Here we should have $a<y<b<c$. After the partial rotation we get a valid right star operation
\[
\tilde{w}=\cdots a \cdots bcy \cdots  \rightsquigarrow  \tilde{w}^*=\cdots a\cdots byc \cdots
\] 
and a valid channel $C' = \ldots, a, y, c, \ldots \subset w^*$ (i.e. replacing $b$ with $y$ in $C$) so that the commutative diagram can be completed.
\item The right star operation looks like:
\[
w= \cdots yab \cdots c \cdots \rightsquigarrow  w^*=\cdots yba \cdots c \cdots.
\] 
Here we should have $a<y<b<c$. However, it is impossible since $y$ and its translates can be inserted to $C\tbr{1}$ to yield a longer stream in $\tilde{w}$.
\end{enumerate}

\subsubsection*{\textbf{Case 3.}} Suppose that $\#([p-1, p+1] \cap X)=3$. However, this is impossible since it means that $w(p-1)<w(p)<w(p+1)$ and the right star operation cannot be applied.
 
We exhausted all the possibilities and this finishes the proof of the Diamond lemma.

%

\section{Proof of Blasiak's conjecture}\label{sec:proof}
\subsection{Same left cell implies same $\fQ$}\label{sec:sameleft}
We are ready to prove the half of Theorem \ref{thm:main}, i.e. $\{w \in \tsc_\lambda \mid \sgn_\fQ(w)=\fQ\}$ is a union of left cells. To this end the following lemma is essential.
\begin{lem}\label{lem:asym} For $v, w \in \extSn$, let us write $v(i,j) = in+v(j)$ and $w(i,j) = in+w(j)$ where $i,j \in \bZ$. For $X \subset [1,n]$, suppose that 
\begin{enumerate}[label=\textnormal{(\alph*)}]
\item the relative orders of $\{v(i,x)\}_{i\in [0,n-1],x\in X}$ and $\{w(i,x)\}_{i\in [0,n-1],x\in X}$ are the same,  i.e. for any $i,i' \in [0,n-1],j, j' \in X$ we have $v(i,j)\leq v(i',j')$ if and only if $w(i,j)\leq w(i',j')$. 
\item $v(y)+n^2< v(x)$ and $w(y)+n^2< w(x)$ for $x \in X$ and $y \in [1,n]-X$, and
\item $\sgn_\fQ(v_{[1,n]-X})=\sgn_\fQ(w_{[1,n]-X})$.
\end{enumerate}Then $\sgn_\fQ(v)=\sgn_\fQ(w)$.
\end{lem}
\begin{proof} First suppose that $X=[1,n]$, so that the conditions (b) and (c) are vacuous. In this case, the statement directly follows from the definition of the sign insertion algorithm. More precisely, first note that the entries which appear in one of $\fP_i$ or $v_i$ (resp. $w_i$) in the sign insertion algorithm are contained in $\{v(i,j)\}_{i\in [0,n-1],j\in [1,n]}$ (resp. $\{w(i,j)\}_{i\in [0,n-1],j\in [1,n]}$) since an entry cannot be bumped more than $n$ times during the process. Thus, one can show by induction on the steps of sign insertion that indeed $\fQ_i$ are the same for $v$ and $w$, using the assumption that the relative orders of $\{v(i,j)\}_{i\in [0,n-1],j\in [1,n]}$ and $\{w(i,j)\}_{i\in [0,n-1],j\in [1,n]}$ are the same.

In general, condition (b) implies that $v(i,y) < v(i',x)$ and $w(i,y) < w(i',x)$ for any $i, i' \in [0,n-1], x \in X,$ and $y \in [1,n]-X$. In other words, $v(i',x)$ and $w(i',x)$ for $i'\in [0,n-1]$ and $x \in X$ behave as if they are very large compared to $v(i,y)$ and $w(i,y)$ for $i \in [0,n-1]$ and $y\in [1,n]-X$. Since condition (c) gives the equality of two $\fQ$'s when the sign insertion algorithm is applied to the ``smaller part'' of $v$ and $w$, respectively, similarly to above one can use induction on the steps of sign insertion to conclude the claim. 
\end{proof}

Let us prove that $\sgn_\fQ(v)=\sgn_\fQ(w)$ when $v$ and $w$ are in the same left cell indexed by $Q \in \RSYT(\lambda)$. By Lemma \ref{lem:sgnqpre}, we may assume that $v$ and $w$ are in the same right cell, say indexed by $P \in \RSYT(\lambda)$. Thus there exist $\vv{\phi}, \vv{\rho} \in \bZ^{l(\lambda)}$ such that $\Phi(v) = (P, Q, \vv{\phi})$ and $\Phi(w) = (P, Q, \vv{\rho})$. 



Let $m$ be the multiplicity of $\lambda_1$ in $\lambda$ and let $\fC_w=(C_1, \ldots, C_m)$ as before. Let $X_i =\{x\in [1,n] \mid (x, w(x)) \in C_i\}$. If we let $X_i = \{a_1, \ldots, a_k\}$ where $a_1<\cdots <a_k$ then by definition of the channel we have $w(a_1)<\cdots <w(a_k)<w(a_1)+n$. Moreover, since proper PR/IPR do not affect $\sgn_\fQ(w)$ and the left cell of $w$, by applying them repeatedly we may assume that $0\ll \rho_1\ll \cdots \ll \rho_m$, which results in 
$w(y)\ll w(x_1) \ll \cdots \ll w(x_m)$ for any $y \in [1,n] -\bigsqcup_{i=1}^m X_i$, $x_1 \in X_1$, $\ldots$, $x_m \in X_m$.


Let us recall the description of AMBC in terms of asymptotic realization, i.e. \cite[Section 7]{cpy18}. We observe the following statement.
\begin{lem} \label{lem:asympAMBC} For $w\in \extSn$ and $X \subset [1,n]$, suppose that $\{(x, w(x)) \mid x \in X+n\bZ\}$ is a channel of $w$. Also assume that $w(y) \ll w(x)$ for any $x \in X$ and $y \in [1,n]-X$. Then $\Phi(w) = (P , Q, \vv{\rho})$ for some $P$ and $\vv{\rho}$ if and only if  $Q_1=X$ (as sets) and $\Phi(w\su{[1,n]-X}) = (P', Q_{\geq 2}, \vv{\rho}')$ for some $P'$ and $\vv{\rho}'$.
\end{lem}
\begin{proof} Here we prove the only if part, but its converse is proved similarly. Note that $Q$ is equal to the recording tableau modulo $n$ where the usual Robinson-Schensted algorithm is applied to an infinite sequence $(w(1), w(2), \ldots)$. (This is implicit in \cite[Section 7]{cpy18} and shown in the proof of \cite[Theorem 10.3]{kp20}.) The condition implies that $w(in+x)>w(y)$ for any $i \in \bZ$, $x\in X$, and $y\in [1,in+x-1]$. Thus when $w(in+x)$ is inserted to the infinite insertion tableau it is always located at the end of the first row. Thus it follows that the first row of $Q$ should contain all the elements in $X$. However, since $\#X$ is equal to the length of the first row of $Q_1$ because $\{(x, w(x)) \mid x \in X+n\bZ\}$ is a channel, it follows that $Q_1=X$.

On the other hand, the bumping process of $w(in+y)$ for $i\in \bN$ and $y \in [1,n]-X$ is almost the same as considering the partial permutation $w\su{[1,n]-X}$, except that at the end of each step it bumps some element of the form $w(jn+x)$ for $j \in \bN$ and $x \in X$. It means that the recording tableau of the infinite word corresponding to $w\su{[1,n]-X}$ is the same as that corresponding to $w$ after the first row is removed. Thus by taking residues modulo $n$ we obtain the second claim.
%
%
%
%
%
\end{proof}

In our situation, by iterating the above lemma we see that $X_i = Q_{m+1-i}$ for $i \in [1,m]$. The most important part of this observation is that $X_1, \ldots, X_m$ only depend on $Q$ whenever $0\ll \rho_1\ll \cdots \ll \rho_m$. More precisely, if we set $\vv{\rho}$ such that $\rho_i+N<\rho_{i+1}$ for $i\in [1,m-1]$ and $\rho_1>N$ for a sufficiently large $N$, then the set of $x$-coordinates of each disjoint channel $C_1, \ldots, C_m$ stabilizes as $N$ grows. Moreover, if we set $X = \bigsqcup_{i=1}^m X_i$, then $w(y)+n^2<w(x)$ for $x\in X$ and $y \in[1,n]-X$, and [$w(i,j) \leq w(i',j')$ for $i,i' \in [0,n-1]$ and $j,j' \in X$] if and only if either [$j, j' \in X_k$ and $j\leq j'$ for some $k \in [1,m]$] or [$j \in X_k$ and $j'\in X_{k'}$ for some $1\leq k<k'\leq m$].

Now assume that both $\vv{\rho}$ and $\vv{\phi}$ satisfy $0\ll \rho_1\ll \cdots \ll \rho_m$ and $0\ll\phi_1\ll \cdots \ll \phi_m$, which is again possible. By the argument above and Lemma \ref{lem:asym}, in order to prove $\sgn_{\fQ}(v) = \sgn_{\fQ}(w)$ it suffices to show $\sgn_{\fQ}(v_{[1,n]-X}) = \sgn_{\fQ}(w_{[1,n]-X})$ where $X = \bigsqcup_{i=1}^m X_i= \bigsqcup_{i=1}^m Q_i$. However, Lemma \ref{lem:asympAMBC} implies that the partial permutations $v\su{[1,n]-X}$ and $w\su{[1,n]-X}$ have the same $Q$ tableau under AMBC. Therefore, the result follows from induction on $n$.

\subsection{Same $\fQ$ and same two-sided cell implies same left cell}

Let us continue with the other half of Theorem \ref{thm:main}, i.e. here we prove that $\sgn_\fQ(w) \neq \sgn_\fQ(v)$ if $w$ and $v$ are in the same two-sided cell but in different left cells. 
Suppose that a partition $\lambda \vdash n$ of length $l=l(\lambda)$ and a tableau $T \in \RSYT(\lambda)$ is given. Let us write $L_i=\sum_{j=i+1}^l \lambda_j$ for $i \in [0,l]$ so that $L_{i-1}-L_{i}=\lambda_i$. For $N \in \bZ$, we consider an element $w_{T,N}\in \extSn$ where
\[ (w_{T,N})_{T_i} = \left(Nn(l-i)+L_{i}+1, Nn(l-i)+L_i+2, \ldots, Nn(l-i)+L_{i-1}\right).\]
Usually $N$ will be a sufficiently large natural number.

When $N \gg n>0$, it is not hard to show that $\sgn_\fP(w_{T,N})_{[L_{i}+1,L_{i-1}]}$ is obtained from $(w_{T,N})_{T_i}$ by applying partial rotations, i.e. there exists $m_i \in [0,n-1]$ such that $\sgn_\fP(w_{T,N})_{[L_{i}+1,L_{i-1}]}=\pr_{T_i}^{m_i}(w_{T,N})_{T_i}$. (This can be proved using induction on the steps of sign insertion.) On the other hand, let $C_i = \{(jn+x, jn+w_{T,N}(x)) \mid x \in T_i, j\in \bZ\}$ be a channel of $w_{T,N}$. Similarly, we regard $\sgn_\fP(w_{T,N})$ as an affine permutation and set $C'_i=\{(jn+x, jn+\sgn_\fP(w_{T,N})(x)) \mid x \in [L_{i}+1,L_{i-1}], j\in \bZ\}$ which becomes channel of $\sgn_\fP(w_{T,N})$. Then it is easy to show that when $0<k\ll N$ we have $\pr_{C'_i}^k(\sgn_\fP(w_{T,N}))=\sgn_\fP(\pr_{C_i}^{k}(w_{T,N}))$, i.e. the partial rotation ``commutes with'' $\sgn_\fP$. (This follows from a similar argument to the proof of Lemma \ref{lem:asym} when $X=[1,n]$.) Therefore, if we set $k_i = n\lambda_i-m_i>0$ for $i \in [1,l]$ and define $\tilde{w}_{T,N} = (\pr_{C_1}^{k_1}\circ \cdots \pr_{C_l}^{k_l})(w_{T,N})$ then it satisfies
\begin{align*}
\sgn_\fP(\tilde{w}_{T,N})_{[L_{i}+1,L_{i-1}]} = & \left(Nn(l-i)+n^2+L_{i-1}+1, \ldots, Nn(l-i)+n^2+L_{i}\right)
\\=&(w_{T,N})_{T_i}+(n^2, n^2, \ldots, n^2).
\end{align*}
(Note that $\sgn_\fP(\tilde{w}_{T,N})$ only depends on $\lambda$ and $N$ but not $T$.)
Moreover, as $N \gg 0$, by applying Lemma \ref{lem:asympAMBC} repeatedly we see that $\tilde{w}_{T,N}$ is contained in the left cell parametrized by $T$, i.e. $\Phi(\tilde{w}_{T,N}) = (P, T, \vv{\rho})$ for some $P$ and $\vv{\rho}$.
\begin{example} Let $T=((3,6,7,9),(4,8,10),(1,5),(2))\in \RSYT(4,3,2,1)$. Then we have
\[w_{T,10}=[102,1,307,204,103,308,309,205,310,206].\]
(Here, $N=10$ is sufficiently large for our purpose.) Direct calculation shows that 
\[\sgn_{\fP}(w_{T,10}) = (1, 103, 112, 206, 214, 215, 318, 319, 320, 327).\]
Here $m_1=5, m_2=2, m_3=1, m_4=0$. Thus if we set $\tilde{w}_{T,10} = (\pr_{T_1}^{35}\circ\pr_{T_2}^{28}\circ\pr_{T_3}^{19}\circ \pr_{T_l}^{10})(w_{T,10})=[193, 101, 390, 295, 202, 397, 398, 296, 399, 304]$ then we have
\[\sgn_\fP(\tilde{w}_{T,10})=[101, 202, 203, 304, 305, 306, 407, 408, 409, 410]\]
as desired. Furthermore, we have 
\ytableausetup{tabloids, boxsize=1.2em}
\[\Phi(\tilde{w}_{T,10}) = \left(\ytableaushort{136{10},259,48,7},\ytableaushort{3679,48{10},15,2}, \begin{pmatrix}155\\88\\39\\10\end{pmatrix}\right),\]
i.e. $\tilde{w}_{T,10}$ is contained in the left cell parametrized by $T$.
\end{example}

We are ready to prove the other half of Theorem \ref{thm:main}. For $w, v \in \extSn$, suppose that $w \in \Gamma_T$ and $v\in \Gamma_S$ where $S, T \in \RSYT(\lambda)$ for some $\lambda \vdash n$ but $S\neq T$. Here we prove that $\sgn_\fQ(w) \neq \sgn_\fQ(v)$. Since we already showed that the elements in the same left cell have the same image under $\sgn_\fQ$, it suffices to assume that $w=\tilde{w}_{T,N}$ and $v=\tilde{w}_{S,N}$ for sufficiently large $N$. 

By the argument above, we have $\sgn_\fP(\tilde{w}_{T,N})=\sgn_\fP(\tilde{w}_{S,N})$, i.e. $\sgn_\fP(\tilde{w}_{T,N})$ does not depend on $T$ but only on $\lambda$ once $N\gg 0$ is fixed. However, the map $w\mapsto (\sgn_{\fP}(w), \sgn_{\fQ}(w))$ is injective by the argument in \cite[p.2333]{bla11}. It follows that $\sgn_\fQ(\tilde{w}_{T,N})=\sgn_\fQ(\tilde{w}_{S,N})$ if and only if $\tilde{w}_{T,N}=\tilde{w}_{S,N}$ if and only if $T=S$. Therefore, we have $\sgn_\fQ(\tilde{w}_{T,N})\neq\sgn_\fQ(\tilde{w}_{S,N})$ as desired.

\subsection{Partial rotations and left cells} 
As a bi-product of our argument, we may connect the elements in the intersection of a right cell and a left cell by (inverse) partial rotation. To this end first we observe the following lemma.
\begin{lem} \label{lem:asymchan} For a partial permutation $w$ and $X \subset [1,n]$, suppose that $w_X$ is a channel of $w$ and $w(y)\ll w(x)$ for $x \in X$ and $y \in [1,n]-X$. Let $X'=\{n+1-x \mid x \in \im w_X \cap [1,n]\}$. If we let $\Phi(w) = (P, Q, -)$ and $\Phi(w_{[1,n]-X}) = (P', Q', -)$, then we have $P=\evac((X')\con \evac(P'))$ and $Q=(X) \con Q'$ where $\evac$ is the affine evacuation defined in \cite{cfkly}.
\end{lem}
\begin{proof} By Lemma \ref{lem:asympAMBC} we have $X=Q_1$. Moreover, if we consider $\rot(w^{-1})$ and $X'$ instead then the conditions of Lemma \ref{lem:asympAMBC} is still valid and $X'$ is equal to the first row of $\evac(P)$ because $\Phi(\rot(w^{-1})) = (\evac(Q), \evac(P), -)$. Since $\rot(w^{-1})_{[1,n]-X'}=\rot((w_{[1,n]-X})^{-1})$, it follows that $\evac(P) = (X')\con \evac(P')$, i.e. $P=\evac((X')\con \evac(P'))$ as desired.
\end{proof}
The main statement of this section is as follows.
\begin{thm} Suppose that $P, Q \in \RSYT(\lambda)$ and $w, \tilde{w} \in (\lc_{P})^{-1} \cap \lc_Q$. Then there exists a sequence $w=w_0, w_1, \ldots, w_k=\tilde{w}$, which all lie in $(\lc_{P})^{-1} \cap \lc_Q$, such that for $i \in [1,k]$ either $w_i=\pr_{S_{i-1}}(w_{i-1})$ or $w_i=\pr^{-1}_{S_{i-1}}(w_{i-1})$ where $S_{i-1}$ is a stream (not necessarily a channel) of $w_{i-1}$.
\end{thm}
\begin{proof} Let $\Phi(w) = (P, Q, \vv{\rho})$ and $\Phi(\tilde{w}) = (P, Q, \vv{\varphi})$ for some $\vv{\rho}=(\rho_1, \rho_2, \ldots)$ and $\vv{\varphi}=(\varphi_1, \varphi_2, \ldots)$. By applying the proper partial rotations to the northeast channels of $w$ and $\tilde{w}$, respectively, we may assume that $\rho_1, \varphi_1 \gg 0$. Then by Lemma \ref{lem:asympAMBC}, we see that $w_{Q_1}$ and $\tilde{w}_{Q_1}$ are channels of $w$ and $\tilde{w}$, respectively. Moreover, Lemma \ref{lem:asymchan} implies that $\im w_{Q_1}= \im \tilde{w}_{Q_1}=\{n+1-x \mid x \in \evac(P)_1\}+n\bZ$. Therefore, by applying proper partial rotation we may assume that $w_{Q_1}=\tilde{w}_{Q_1}$ while $w(y)\ll w(x)$ and $\tilde{w}(y)\ll \tilde{w}(x)$ for $x \in X$ and $y \in [1,n]-X$. Now again by Lemma \ref{lem:asymchan}, it suffices to show that $w_{[1,n]-X}$ and $\tilde{w}_{[1,n]-X}$ are connected by a similar sequence described in the statement. Thus the claim follows from induction on the length of $P$ and $Q$.
\end{proof}




\section{Sign insertion and Lascoux-Sch\"utzenberger standardization} \label{sec:LS}
Let us set $\Upsilon_\fQ=\{T\in \RSYT(n) \mid \sgn_\fQ(w)=\fQ \textnormal{ for some } w \in \lc_T\}$. Then Theorem \ref{thm:main} implies that $\Upsilon_\fQ \cap \RSYT(\lambda)$ is a singleton if nonempty. Our goal in this section is to provide a relation between elements in $\Upsilon_{\fQ}$, necessarily in different shapes.

\subsection{Tableaux, crystal operators, and Robinson-Schensted algorithm}

For $\alpha \models n$ and a row-standard Young tabloid $T \in \RSYT(\alpha)$ we consider its associated two-row array $\cA(T)$ whose first row records $\textnormal{(length of $\alpha$)}+1-\textnormal{(row number)}$ and whose second row is $\rw(T)$. For example, if $T=((2),\emptyset,(3,5,6),(1,4))$ then $\cA(T)=\begin{pmatrix} 1&1&2&2&2&4\\1&4&3&5&6&2\end{pmatrix}$. Then we may apply the usual Robinson-Schensted-Knuth algorithm (see \cite[Chapter 7]{sta86} for more detail) to $\cA(T)$. Let us denote this map by $\RSK: \RSYT(\alpha) \rightarrow \bigsqcup_{\lambda \vdash n} \SYT(\lambda) \times \SSYT(\lambda, \alpha^{rev}): T \mapsto (\RSK_P(T), \RSK_Q(T))$. This is indeed a bijection.

Let $\Sinf$ be the set of permutations of $\bZ_{>0}$ of finite support. (In other words, $\Sinf = \bigcup_{n\geq 1}\Sym_n$.) Let $s_i \in \Sinf$ be the transposition swapping $i$ and $i+1$. For a composition $\alpha \models n$, we may regard it as an infinite sequence all of whose entries but finite are zero, say $\alpha=(\alpha_1, \alpha_2, \ldots)$. Then $s_i\cdot \alpha=(\alpha_{1}, \alpha_{2},\ldots, \alpha_{i-1}, \alpha_{i+1}, \alpha_i, \alpha_{i+2}, \ldots)$ is well-defined and there exists a bijection $s_i: \SSYT(n, \alpha) \rightarrow \SSYT(n, s_i \cdot \alpha)$ such that $\rw(T) \mapsto \rw(s_i(T))$ is a crystal reflection operator defined by Lascoux and Sch\"utzenberger \cite{ls81} (see also \cite[3.2]{sw00}). These operations are involutions and satisfy the braid relation, and thus these yield an action of $\Sinf$ on $\bigsqcup_{\alpha \models n} \SSYT(n, \alpha)$.

Moreover, there exists a bijection $R_i : \RSYT(\alpha) \rightarrow \RSYT(s_i \cdot \alpha)$, called a combinatorial $R$-matrix, which comes from an isomorphism of tensor products of one-row Kirillov-Reshetikhin crystals. (See \cite[4.8]{shi05} or \cite[Definition 3.11]{cfkly}.) These operations are also involutions and satisfy the braid relation, and thus these yield an action of $\Sinf$ on $\bigsqcup_{\alpha \models n}\RSYT(\alpha)$, denoted by $R_\sigma$ for $\sigma \in \Sinf$.
Then by \cite[Proposition 5.1]{shi05} (see also \cite[Proposition 3.21]{cfkly} and the footnote thereafter), we have $\RSK_P(R_i(T)) = \RSK_P(T) $ and $\RSK_Q(R_i(T))  = s_{l(T)-i}\RSK_Q(T)$. Therefore, if $\sigma \in \Sym_{l(T)}$ then $\RSK_P(R_\sigma(T)) = \RSK_P(T) $ and $\RSK_Q(R_\sigma(T))  = (w_0\sigma w_0)\RSK_Q(T)$ where $w_0$ is the longest element of $\Sym_{l(T)}$. 

\subsection{Lascoux-Sch\"utzenberger standardization}
We recall the standardization map of Lascoux and Sch\"utzenberger. Here we mainly follow the argument and notations from \cite{sw00}. Suppose that we are given two compositions $\alpha, \beta \models n$ whose rearrangements are partitions $\alpha^+ , \beta^+\vdash n$, respectively. Assume that $\alpha^+ \geq \beta^+$ with respect to dominance order. Then there exists an injective map $\theta_\alpha^\beta : \SSYT(n, \alpha) \rightarrow \SSYT(n, \beta)$ defined as follows.
\begin{enumerate}[label=\textbullet]
\item If $\alpha^+=\beta^+$ then choose any permutation $\sigma\in \Sinf$ so that $\sigma\cdot\alpha = \beta$ and let $\theta_\alpha^\beta=\sigma$.
\item If $\beta_i=\alpha_i$ for $i \geq 3$ and $\beta_1+1=\alpha_1>\beta_2=\alpha_2+1$, then $\theta_\alpha^\beta$ acts on a tableau by changing the rightmost letter 1 to 2 (which is always possible).
\item In general, there exists a sequence $\alpha=\alpha_0, \alpha_1, \ldots, \alpha_k = \beta$ where each $\alpha_i \rightarrow \alpha_{i+1}$ is either in the first or the second case. Then we define $\theta_\alpha^\beta=\theta^{\alpha_k}_{\alpha_{k-1}}\circ \cdots \circ \theta^{\alpha_1}_{\alpha_{0}}$.
\end{enumerate}
Then by \cite{las91} this map is well-defined (i.e. it does not depend on the choice of $\alpha=\alpha_0, \alpha_1, \ldots, \alpha_k = \beta$ on the last part) and also satisfies $\theta_{\beta}^{\gamma}\circ\theta_{\alpha}^{\beta}=\theta_\alpha^\gamma$.

For $\alpha \models n$, note that $\SSYT(n, \alpha)$ is endowed with a graded poset structure by cyclage where the grading is given by cocharge. By \cite{ls81} and \cite{las91}, the map $\theta_{\alpha}^\beta: \SSYT(n, \alpha)\rightarrow \SSYT(n, \beta)$ is indeed a grade-preserving poset embedding (see also \cite[Theorem 44, 45]{sw00}).

%
%
%

\subsection{Investigation of $\Upsilon_\fQ$}
The main result of this section is the following theorem.
\begin{thm} \label{thm:lsstd}Suppose that $\Upsilon_\fQ=\{T\in \RSYT(n) \mid \sgn_\fQ(w)=\fQ \textnormal{ for some } w \in \lc_T\}$ is nonempty. Then,
\begin{enumerate}[label=\textbullet]
\item $\RSK_P(\Upsilon_\fQ)$ consists of a single element. 
\item If $S, T \in \Upsilon_{\fQ}$ and $\sh(S)\geq \sh(T)$ then $\theta_{\sh(S)^{rev}}^{\sh(T)^{rev}}(\RSK_Q(S))=\RSK_Q(T)$. Moreover, if $T \in \Upsilon_{\fQ}$ and $\sh(T) \geq \lambda$ for some partition $\lambda \vdash n$ then $\theta_{\sh(T)^{rev}}^{\lambda^{rev}}(\RSK_Q(T)) \in \RSK_Q(\Upsilon_{\fQ})$. In other words, $\RSK_Q(\Upsilon_{\fQ})$ is ``stable under $\theta$''.
\end{enumerate}
\end{thm}
As a corollary we obtain the following.
\begin{cor} For each $n \in \bZ_{>0}$, $\{\sgn_\fQ(w) \mid w \in \extSn\}$ is canonically bijective with $\RSYT(1^n)$.
\end{cor}
\begin{proof} If we define $\RSYT(1^n) \rightarrow \{\sgn_\fQ(w) \mid w \in \extSn\} : T \mapsto \sgn_\fQ(v_T)$ where $v_T$ is any element in $\lc_T$, then this map is well-defined and injective by Theorem \ref{thm:main}. We claim that this is also surjective. Indeed, suppose that $\Upsilon_\fQ \neq \emptyset$ and $T\in \Upsilon_\fQ$. Then Theorem \ref{thm:lsstd} implies that $\RSK_P(\Upsilon_\fQ) = \{\RSK_P(T)\}$ and $\RSK_Q(\Upsilon_\fQ) \owns \theta_{\sh(T)^{rev}}^{(1^n)}(\RSK_Q(T))$. Now if we let $S \in \RSYT(1^n)$ be the preimage of $(\RSK_P(T), \theta_{\sh(T)^{rev}}^{(1^n)}(\RSK_Q(T)))$ under RSK then we see that $\Upsilon_\fQ \owns S$ as desired.
\end{proof}
\begin{example} When $n=4$, the sets $\Upsilon_\fQ$ are as follows.
\begin{align*}
\Upsilon_{(1, 2, 3, 4)}&=(((1, 2, 3, 4)),((2, 3, 4), (1)),((3, 4), (1, 2)),((3, 4), (2), (1)),((4), (3), (2), (1)))\allowdisplaybreaks\\
\Upsilon_{(1, 2, 3, 5)}&=(((1, 2, 3), (4)),((2, 3), (1, 4)),((2, 3), (4), (1)),((3), (4), (2), (1)))\allowdisplaybreaks\\
\Upsilon_{(1, 2, 4, 5)}&=(((1, 2, 4), (3)),((2, 4), (1, 3)),((2, 4), (3), (1)),((2), (4), (3), (1)))\allowdisplaybreaks\\
\Upsilon_{(1, 3, 4, 5)}&=(((1, 3, 4), (2)),((1, 4), (2, 3)),((1, 4), (3), (2)),((1), (4), (3), (2)))\allowdisplaybreaks\\
\Upsilon_{(1, 3, 5, 6)}&=(((1, 3), (2, 4)),((1, 3), (4), (2)),((3), (1), (4), (2)))\allowdisplaybreaks\\
\Upsilon_{(1, 2, 5, 6)}&=(((1, 2), (3, 4)),((1, 2), (4), (3)),((2), (1), (4), (3)))\allowdisplaybreaks\\
\Upsilon_{(1, 4, 5, 7)}&=(((1, 4), (2), (3)),((1), (2), (4), (3)))\allowdisplaybreaks\\
\Upsilon_{(1, 3, 5, 7)}&=(((1, 3), (2), (4)),((1), (3), (4), (2)))\allowdisplaybreaks\\
\Upsilon_{(1, 2, 5, 7)}&=(((1, 2), (3), (4)),((2), (3), (4), (1)))\allowdisplaybreaks\\
\Upsilon_{(1, 2, 3, 6)}&=(((2, 3), (1), (4)),((3), (2), (4), (1)))\allowdisplaybreaks\\
\Upsilon_{(1, 2, 4, 6)}&=(((2, 4), (1), (3)),((4), (2), (3), (1)))\allowdisplaybreaks\\
\Upsilon_{(1, 3, 4, 6)}&=(((3, 4), (1), (2)),((4), (1), (3), (2)))\allowdisplaybreaks\\
\Upsilon_{(1, 4, 7, 9)}&=(((4), (1), (2), (3)))
 \qquad\qquad\qquad \Upsilon_{(1, 3, 7, 9)}=(((3), (1), (2), (4)))\allowdisplaybreaks\\
\Upsilon_{(1, 2, 7, 9)}&=(((2), (1), (3), (4)))
\qquad\qquad\qquad \Upsilon_{(1, 2, 3, 7)}=(((3), (2), (1), (4)))\allowdisplaybreaks\\
\Upsilon_{(1, 2, 4, 7)}&=(((4), (2), (1), (3)))
\qquad\qquad\qquad \Upsilon_{(1, 3, 4, 7)}=(((4), (3), (1), (2)))\allowdisplaybreaks\\
\Upsilon_{(1, 3, 6, 8)}&=(((3), (4), (1), (2)))
\qquad\qquad\qquad \Upsilon_{(1, 2, 6, 8)}=(((2), (4), (1), (3)))\allowdisplaybreaks\\
\Upsilon_{(1, 4, 5, 8)}&=(((1), (4), (2), (3)))
\qquad\qquad\qquad \Upsilon_{(1, 5, 8, 10)}=(((1), (2), (3), (4)))\allowdisplaybreaks\\
\Upsilon_{(1, 3, 5, 8)}&=(((1), (3), (2), (4)))
\qquad\qquad\qquad \Upsilon_{(1, 2, 5, 8)}=(((2), (3), (1), (4)))
\end{align*}
If we apply the RSK to the elements in each set, then we get the following.
\begin{align*}
\RSK_P(\Upsilon_{(1, 2, 3, 4)})&=\{((1, 2, 3, 4))\}
&\RSK_P(\Upsilon_{(1, 2, 3, 5)})&=\{((1, 2, 3), (4))\}\allowdisplaybreaks\\
\RSK_P(\Upsilon_{(1, 2, 4, 5)})&=\{((1, 2, 4), (3))\}
&\RSK_P(\Upsilon_{(1, 3, 4, 5)})&=\{((1, 3, 4), (2))\}\allowdisplaybreaks\\
\RSK_P(\Upsilon_{(1, 3, 5, 6)})&=\{((1, 3), (2, 4))\}
&\RSK_P(\Upsilon_{(1, 2, 5, 6)})&=\{((1, 2), (3, 4))\}\allowdisplaybreaks\\
\RSK_P(\Upsilon_{(1, 4, 5, 7)})&=\{((1, 4), (2), (3))\}
&\RSK_P(\Upsilon_{(1, 3, 5, 7)})&=\{((1, 3), (2), (4))\}\allowdisplaybreaks\\
\RSK_P(\Upsilon_{(1, 2, 5, 7)})&=\{((1, 2), (3), (4))\}
&\RSK_P(\Upsilon_{(1, 2, 3, 6)})&=\{((1, 2, 3), (4))\}\allowdisplaybreaks\\
\RSK_P(\Upsilon_{(1, 2, 4, 6)})&=\{((1, 2, 4), (3))\}
&\RSK_P(\Upsilon_{(1, 3, 4, 6)})&=\{((1, 3, 4), (2))\}\allowdisplaybreaks\\
\RSK_P(\Upsilon_{(1, 4, 7, 9)})&=\{((1, 4), (2), (3))\}
&\RSK_P(\Upsilon_{(1, 3, 7, 9)})&=\{((1, 3), (2), (4))\}\allowdisplaybreaks\\
\RSK_P(\Upsilon_{(1, 2, 7, 9)})&=\{((1, 2), (3), (4))\}
&\RSK_P(\Upsilon_{(1, 2, 3, 7)})&=\{((1, 2, 3), (4))\}\allowdisplaybreaks\\
\RSK_P(\Upsilon_{(1, 2, 4, 7)})&=\{((1, 2, 4), (3))\}
&\RSK_P(\Upsilon_{(1, 3, 4, 7)})&=\{((1, 3, 4), (2))\}\allowdisplaybreaks\\
\RSK_P(\Upsilon_{(1, 3, 6, 8)})&=\{((1, 3), (2, 4))\}
&\RSK_P(\Upsilon_{(1, 2, 6, 8)})&=\{((1, 2), (3, 4))\}\allowdisplaybreaks\\
\RSK_P(\Upsilon_{(1, 4, 5, 8)})&=\{((1, 4), (2), (3))\}
&\RSK_P(\Upsilon_{(1, 5, 8, 10)})&=\{((1), (2), (3), (4))\}\allowdisplaybreaks\\
\RSK_P(\Upsilon_{(1, 3, 5, 8)})&=\{((1, 3), (2), (4))\}
&\RSK_P(\Upsilon_{(1, 2, 5, 8)})&=\{((1, 2), (3), (4))\}
\end{align*}
\begin{align*}
\RSK_Q(\Upsilon_{(1, 2, 3, 4)})&=\{((1, 1, 1, 1)),((1, 2, 2, 2)),((1, 1, 2, 2)),((1, 2, 3, 3)),((1, 2, 3, 4))\}\allowdisplaybreaks\\
\RSK_Q(\Upsilon_{(1, 2, 3, 5)})&=\{((1, 2, 2), (2)),((1, 1, 2), (2)),((1, 2, 3), (3)),((1, 2, 3), (4))\}\allowdisplaybreaks\\
\RSK_Q(\Upsilon_{(1, 2, 4, 5)})&=\{((1, 2, 2), (2)),((1, 1, 2), (2)),((1, 2, 3), (3)),((1, 2, 3), (4))\}\allowdisplaybreaks\\
\RSK_Q(\Upsilon_{(1, 3, 4, 5)})&=\{((1, 2, 2), (2)),((1, 1, 2), (2)),((1, 2, 3), (3)),((1, 2, 3), (4))\}\allowdisplaybreaks\\
\RSK_Q(\Upsilon_{(1, 3, 5, 6)})&=\{((1, 1), (2, 2)),((1, 2), (3, 3)),((1, 2), (3, 4))\}\allowdisplaybreaks\\
\RSK_Q(\Upsilon_{(1, 2, 5, 6)})&=\{((1, 1), (2, 2)),((1, 2), (3, 3)),((1, 2), (3, 4))\}\allowdisplaybreaks\\
\RSK_Q(\Upsilon_{(1, 4, 5, 7)})&=\{((1, 3), (2), (3)),((1, 2), (3), (4))\}\allowdisplaybreaks\\
\RSK_Q(\Upsilon_{(1, 3, 5, 7)})&=\{((1, 3), (2), (3)),((1, 2), (3), (4))\}\allowdisplaybreaks\\
\RSK_Q(\Upsilon_{(1, 2, 5, 7)})&=\{((1, 3), (2), (3)),((1, 2), (3), (4))\}\allowdisplaybreaks\\
\RSK_Q(\Upsilon_{(1, 2, 3, 6)})&=\{((1, 3, 3), (2)),((1, 2, 4), (3))\}\allowdisplaybreaks\\
\RSK_Q(\Upsilon_{(1, 2, 4, 6)})&=\{((1, 3, 3), (2)),((1, 2, 4), (3))\}\allowdisplaybreaks\\
\RSK_Q(\Upsilon_{(1, 3, 4, 6)})&=\{((1, 3, 3), (2)),((1, 2, 4), (3))\}\allowdisplaybreaks\\
\RSK_Q(\Upsilon_{(1, 4, 7, 9)})&=\{((1, 4), (2), (3))\}
\qquad\qquad\qquad \RSK_Q(\Upsilon_{(1, 3, 7, 9)})=\{((1, 4), (2), (3))\}\allowdisplaybreaks\\
\RSK_Q(\Upsilon_{(1, 2, 7, 9)})&=\{((1, 4), (2), (3))\}
\qquad\qquad\qquad \RSK_Q(\Upsilon_{(1, 2, 3, 7)})=\{((1, 3, 4), (2))\}\allowdisplaybreaks\\
\RSK_Q(\Upsilon_{(1, 2, 4, 7)})&=\{((1, 3, 4), (2))\}
\qquad\qquad\qquad \RSK_Q(\Upsilon_{(1, 3, 4, 7)})=\{((1, 3, 4), (2))\}\allowdisplaybreaks\\
\RSK_Q(\Upsilon_{(1, 3, 6, 8)})&=\{((1, 3), (2, 4))\}
\qquad\qquad\qquad \RSK_Q(\Upsilon_{(1, 2, 6, 8)})=\{((1, 3), (2, 4))\}\allowdisplaybreaks\\
\RSK_Q(\Upsilon_{(1, 4, 5, 8)})&=\{((1, 3), (2), (4))\}
\qquad\qquad\qquad \RSK_Q(\Upsilon_{(1, 5, 8, 10)})=\{((1), (2), (3), (4))\}\allowdisplaybreaks\\
\RSK_Q(\Upsilon_{(1, 3, 5, 8)})&=\{((1, 3), (2), (4))\}
\qquad\qquad\qquad \RSK_Q(\Upsilon_{(1, 2, 5, 8)})=\{((1, 3), (2), (4))\}
\end{align*}
Here $\RSK_P(\Upsilon_{\fQ})$ is a singleton and the elements in each $\RSK_Q(\Upsilon_{\fQ})$ are connected by the standardization map $\theta$. For example, for $\RSK_Q(\Upsilon_{(1,2,3,6)})$ we have
\ytableausetup{notabloids}
\[
\ytableaushort{133,2} \xmapsto{\theta_{(1,1,2)}^{(2,0,1,1)}}\ytableaushort{114,3}  \xmapsto{\theta_{(2,1,1)}^{(1,1,1,1)}}\ytableaushort{124,3}.
\]
\end{example}

For the proof of Theorem \ref{thm:lsstd}, we generalize the definition of $w_{T,N}$ to the case when $T$ is a row-standard Young tabloid that is not necessarily a tableau. The formula is identical; for a composition $\alpha=\sh(T)$ of length $l=l(\alpha)$ let us write $L_i=\sum_{j=i+1}^l \alpha_j$ for $i \in [0,l]$ so that $L_{i-1}-L_{i}=\alpha_i$. Then for $N \in \bZ$ we set $w_{T,N}\in \extSn$ so that
\[ (w_{T,N})_{T_i} = \left(Nn(l-i)+L_{i}+1, Nn(l-i)+L_i+2, \ldots, Nn(l-i)+L_{i-1}\right).\]
We claim the following.
\begin{lem}\label{lem:Rmat}
Suppose that $S$ and $T$ are row-standard Young tabloids such that $S=R_t(T)$ for some combinatorial $R$-matrix $R_t$. If $N \gg 0$ then $w_{S,N}$ and $w_{T,N}$ are in the same left cell.
\end{lem}
\begin{proof} We first claim that $w_{S,N}$ and $w_{T,N}$ are contained in the same two-sided cell parametrized by the rearrangement of $\sh(S)$. By the construction of $w_{S,N}$, for $i, j \in [1,n]$ we have $i<j$ in the Shi poset of ${w_{S,N}}$ if and only if $j$ is in the upper row of $S$ than $i$. In this case, it is easy to show that  the corresponding Greene-Kleitman partition is given by the rearrangement of $\sh(S)$ as desired. Since the same result holds for $w_{T,N}$, we see that $w_{S,N}$ and $w_{T,N}$ are in the same two-sided cell.

Thus, by Theorem \ref{thm:main} it suffices to show that $\sgn_\fQ(w_{S,N})=\sgn_{\fQ}(w_{T,N})$. Since $S_i=T_i$ for $1\leq i<t$, from Lemma \ref{lem:asym} it suffices to assume that $t=1$. Moreover, if we again apply Lemma \ref{lem:asym} to $\rot(w_{S,N})$ and $\rot(w_{T,N})$, respectively, then indeed it suffices to assume that $S$ and $T$ are two-row tabloids.

Without loss of generality we may assume that $S_1$ is longer than $S_2$. Then $w_{S,N} \in \lc_S$ as before. It is also clear that $\rot(w_{T,N}) \in \lc_{T^c}$ where $T^c$ is the tableau satisfying $T^c_1 = \{n+1-x \mid x \in T_2\}$ and $T^c_2 = \{n+1-x \mid x \in T_1\}$ as sets. By \cite[Theorem 3.18]{cfkly}, it follows that $w_{T,N}$ is contained in the cell parametrized by $R_1(T)=S$ (i.e. the image of $T^c$ under the affine evacuation) as desired.
\end{proof}

The lemma below is another ingredient for the proof of Theorem \ref{thm:lsstd}.
\begin{lem}\label{lem:f1} Let $S =(S_1, \ldots, S_l) \in \RSYT(\alpha)$ where $\alpha=(\alpha_1, \ldots, \alpha_l)$ is a composition of $n$ of length $l$, $S_l=(x_1, \ldots, x_{a_l}),$ and $\alpha_{l-1}=0$. Let $T$ be a row-standard Young tabloid where $T_i=S_i$ for $i\not\in \{l-1,l\}$, $T_{l-1}=(x_{a_l})$, and $T_l=(x_1, \ldots, x_{a_l-1})$. If $N \gg 0$ then we have $\sgn_\fQ(w_{S, N})=\sgn_\fQ(w_{T,N})$.
\end{lem}
\begin{proof}
We set $U=(U_1, \ldots, U_l)$ to be the row-standard Young tabloid such that $U_i=T_i=S_i$ for $i\not\in \{l-1,l\}$, $U_{l-1}=(x_2, \ldots, x_{a_l})$, and $U_l=(x_1)$. Note that $R_{l-1}(U) = T$ which means that $w_{T,N}$ and $w_{U,N}$ have an identical image under $\sgn_\fQ$ by Theorem \ref{lem:Rmat}. Direct calculation shows that
\[w_{U,N} = w_{S,N}+\sum_{i=2}^{a_l}Nn\vv{e}_{i}=w_{S,N}+\sum_{i=1}^{a_l}Nn\vv{e}_{i}-Nn\vv{e}_1\]
where $\vv{e}_i$ is the standard vector with 1 on the $i$-th coordinate. 

It follows from Lemma \ref{lem:asym} that $\sgn_{\fQ}(w_{S,N})=\sgn_{\fQ}(w_{S,N}+\sum_{i=1}^{a_l}Nn\vv{e}_{i})$ if we set $X=[1,n]-S_l$. Now the element $w_{U,N}$ is obtained from $w_{S,N}+\sum_{i=1}^{a_l}Nn\vv{e}_{i}$ by replacing $1+Nn$ on the first coordinate with $1$. However, as $1+Nn$ was already the smallest element of $w_{S,N}+\sum_{i=1}^{a_l}Nn\vv{e}_{i}$, lowering this entry does not affect the sign insertion process as it is inserted at the first step and not bumped by any other elements, which implies that $\sgn_\fQ(w_{U,N}) = \sgn_\fQ(w_{S,N}+\sum_{i=1}^{a_l}Nn\vv{e}_{i})$. Thus the result follows.
\end{proof}

Now we are ready to prove Theorem \ref{thm:lsstd}.
\begin{proof}[Proof of Theorem \ref{thm:lsstd}]
Suppose that $T \in \Upsilon_\fQ \cap \RSYT(\lambda)$ for some $\lambda \vdash n$ and assume that $(P,Q) = (\RSK_P(T), \RSK_Q(T))$. We regard $\lambda$ as a composition of length $n$ (with $n-l(\lambda)$ zeroes at the end) and let $Q' = \theta_{\lambda^{rev}}^{(1^n)}(Q)$. Then there exists a path $\lambda = \alpha_0, \alpha_1, \ldots, \alpha_k =(1^n)$ where each $\alpha_i$ is a composition of $n$ and $\alpha_i \mapsto \alpha_{i+1}$ is either rearrangement of parts or of the form $(0,\ldots, 0, a, 0, \ldots) \mapsto (0,\ldots, 0, 1,  a-1,\ldots)$. (Here we allow rearrangement which permutes nonzero parts with zero parts but this does not cause any problem.) Then we may write $Q' = (\theta_{\alpha_{k-1}^{rev}}^{(1^n)} \circ \theta_{\alpha_{k-2}^{rev}}^{\alpha_{k-1}^{rev}} \circ \cdots \circ \theta_{\alpha_{1}^{rev}}^{\alpha_{2}^{rev}}\circ \theta_{\lambda^{rev}}^{\alpha_{1}^{rev}})(Q)$.

Let us write $T=T_0, T_1, \ldots, T_k$ where $(\RSK_P(T_i), \RSK_Q(T_i))=(P, (\theta_{\alpha_{i-1}^{rev}}^{\alpha_{i}^{rev}} \circ \cdots \circ \theta_{\alpha_{1}^{rev}}^{\alpha_{2}^{rev}}\circ \theta_{\lambda^{rev}}^{\alpha_{1}^{rev}})(Q))$. We claim that $T_k \in \Upsilon_\fQ$ by successively showing $\sgn_{\fQ}(w_{T_i,N}) = \fQ$ for $N \gg0$. The $i=0$ case is trivial. Now if $\alpha_{i+1}=\sigma \cdot \alpha_i$ for some $\sigma \in \Sym_n$, then as we observed above we have $T_{i+1} = w_0\sigma w_0\cdot T_i$ where $w_0$ is the longest element of $\Sym_n$, thus the claim follows from Lemma \ref{lem:Rmat}. In the other case, first note that the shapes of $T_{i+1}$ and $T_i$ are $(\ldots,1,a-1, 0, \ldots,0)$ and $(\ldots,0,a, 0, \ldots,0)$ for some $a>0$, respectively. Now direct calculation shows that $T_{i+1}$ is obtained from $T_i$ by the similar process to Lemma \ref{lem:f1}, i.e. changing $T_i=(\ldots, \emptyset, (x_1, \ldots, x_k), \emptyset, \ldots, \emptyset)$ to $T_{i+1}=(\ldots,(x_k) , (x_1, \ldots, x_{k-1}), \emptyset, \ldots, \emptyset)$. Thus, by Lemma \ref{lem:f1} the claim also follows.

As a result, for any $T \in \Upsilon_\fQ$ such that $(P,Q) = (\RSK_P(T), \RSK_Q(T))$ there exists $T' \in \Upsilon_\fQ \cap \RSYT(1^n)$ such that $(P, \theta_{\sh(T)^{rev}}^{(1^n)}(Q)) = (\RSK_P(T'), \RSK_Q(T'))$. Now suppose that $S \in \Upsilon_\fQ$ and set $S' \in \Upsilon_\fQ \cap \RSYT(1^n)$ similarly. As $\Upsilon_\fQ \cap \RSYT(1^n)$ is a singleton by Theorem \ref{thm:main}, we should have $S'=T'$ which implies that $\RSK_P(T) = \RSK_P(T') = \RSK_P(S') = \RSK_P(S)$ and $\theta_{\sh(T)^{rev}}^{(1^n)}(\RSK_Q(T)) =\RSK_Q(T')=\RSK_Q(S')= \theta_{\sh(S)^{rev}}^{(1^n)}(\RSK_Q(S))$. Now if $\sh(S)\geq \sh(T)$ with respect to dominance order then $(\theta_{\sh(T)^{rev}}^{(1^n)}\circ\theta_{\sh(S)^{rev}}^{\sh(T)^{rev}})(\RSK_Q(S))=\theta_{\sh(S)^{rev}}^{(1^n)}(\RSK_Q(S))=\theta_{\sh(T)^{rev}}^{(1^n)}(\RSK_Q(T))$ which implies that $\theta_{\sh(S)^{rev}}^{\sh(T)^{rev}}(\RSK_Q(S))=\RSK_Q(T)$ as $\theta_{\sh(S)^{rev}}^{\sh(T)^{rev}}$ is an injection (it is an embedding of a poset).

Finally, suppose that we are given $T \in \Upsilon_\fQ$ such that $(P,Q) = (\RSK_P(T), \RSK_Q(T))$ and $\sh(T) \geq \lambda$ for some partition $\lambda\vdash n$. Let $S \in \RSYT(n, \lambda^{rev})$ be such that $(\RSK_P(S), \RSK_Q(S)) = (P, \theta_{\sh(T)^{rev}}^{\lambda^{rev}}(Q))$ and suppose that $S \in \Upsilon_{\tilde{\fQ}}$ for some $\tilde{\fQ}$. We set $T' \in \Upsilon_\fQ \cap \RSYT(1^n)$ and $S' \in \Upsilon_{\tilde{\fQ}} \cap \RSYT(1^n)$ as above. Then $\RSK_P(T') = \RSK_P(T)=\RSK_P(S)=\RSK_P(S')$ and $\RSK_Q(T') = \theta_{\sh(T)^{rev}}^{(1^n)}(\RSK_Q(T)) = (\theta_{\lambda^{rev}}^{(1^n)}\circ\theta_{\sh(T)^{rev}}^{\lambda^{rev}})(\RSK_Q(T)) = \theta_{\lambda^{rev}}^{(1^n)}(\RSK_Q(S)) = \RSK_Q(S')$. Since the RSK algorithm is injective it follows that $S'=T'$ which implies that $\fQ=\tilde{\fQ}$. Thus we have $S \in \Upsilon_\fQ$. It suffices for the proof.
\end{proof}

\bibliographystyle{amsalpha}
\bibliography{blasiak}

\end{document}